\documentclass[10pt, a4paper]{article}
\usepackage[a4paper, total={6in, 9in}]{geometry}

\pdfoutput=1

\usepackage{algorithm}
\usepackage{algorithmic}
\usepackage[vlined,lines numbered,ruled,algo2e]{algorithm2e}
\usepackage{amsmath}
\usepackage{amssymb}
\usepackage{amsthm}
\usepackage{epsfig}
\usepackage{amsfonts}
\usepackage{geometry}
\usepackage{hyperref}
\usepackage{url}
\usepackage{array}
\usepackage{color}
\usepackage{gauss} 
\usepackage{graphicx}
\usepackage{subfig}
\usepackage{grffile}
\usepackage{enumerate}
\usepackage{tikz}
\usetikzlibrary{patterns}

\newtheorem{theorem}{Theorem}[section]
\newtheorem{lemma}[theorem]{Lemma}
\newtheorem{remark}[theorem]{Remark}

\newcommand{\nb}{{\mathsf{nb}}}

\newcommand{\R}{{\mathbb R}}

\newcommand{\blackx}{\mathsf x}
\newcommand{\blueo}{\textcolor{blue}{\mathsf o}}
\newcommand{\blacko}{\mathsf o}
\newcommand{\redx}{\textcolor{red}{\mathsf x}}

\let\oldparagraph\paragraph
\renewcommand\paragraph[1]{\oldparagraph{#1.}}

\newcommand{\change}[1]{{#1}}
\newcommand{\newchange}[1]{{#1}}

\newcommand{\TheTitle}{A Householder-based algorithm for Hessenberg-triangular reduction} 

\begin{document}

\title{{\TheTitle}\thanks{
      \change{The first author} has received financial support from the SNSF research project \emph{Low-rank updates of matrix functions and fast eigenvalue solvers} and the Croatian Science Foundation grant HRZZ-9345.
      \change{The second author} has received financial support from the European Union's Horizon 2020 research and innovation programme under the NLAFET grant agreement No~671633.
    }}

\author{
  Zvonimir Bujanovi\'{c}\thanks{Department of Mathematics, Faculty of Science, University of Zagreb, Zagreb, Croatia (\href{mailto:zbujanov@math.hr}{zbujanov@math.hr}).}
  \and
  Lars Karlsson\thanks{Department of Computing Science, Ume\aa{} University, Ume\aa{}, Sweden (\href{mailto:larsk@cs.umu.se}{larsk@cs.umu.se}).}
  \and
  Daniel Kressner\thanks{Institute of Mathematics, EPFL, Lausanne, Switzerland (\href{mailto:daniel.kressner@epfl.ch}{daniel.kressner@epfl.ch}.}
}

\date{} % October 23rd, 2017

%%%%%%%%%%%%%%%%%%%%%%%%%%%%%%%%%%%%%%%%%%%%%%%%%%%%%%%%%%%%%%%%%%%%%%%%%%%%%%%%

\maketitle

\begin{abstract}
  The QZ algorithm for computing eigenvalues and eigenvectors of a matrix pencil $A - \lambda B$ requires that the matrices first be reduced to Hessenberg-triangular (HT) form.
  The current method of choice for HT reduction relies entirely on Givens rotations \change{regrouped and} accumulated into small dense matrices which are subsequently applied using matrix multiplication routines.
  A non-vanishing fraction of the total flop count must nevertheless still be performed as sequences of overlapping Givens rotations \change{alternately} applied from the left and from the right.
  The many data dependencies associated with this computational pattern leads to inefficient use of the processor and \change{poor scalability}.
  In this paper, we therefore introduce a fundamentally different approach that relies entirely on (large) Householder reflectors partially accumulated into \change{block reflectors, by using} (compact) WY representations.
  Even though the new algorithm requires more floating point operations than the state of the art algorithm, extensive experiments on both real and synthetic data indicate that it is still competitive, even in a sequential setting.
  The new algorithm is conjectured to have better parallel scalability, an idea which is partially supported by early small-scale experiments using multi-threaded BLAS.
  The design and evaluation of a parallel formulation is future work.
\end{abstract}

%!TEX root = 0__main.tex

\section{Introduction}

Given two matrices $A,B \in \R^{n\times n}$ the QZ algorithm proposed by Moler and Stewart~\cite{Moler1973} for computing eigenvalues and eigenvectors of the matrix pencil $A-\lambda B$ consists of three steps.
First, a QR or an RQ factorization is performed to reduce $B$ to triangular form.
Second, a Hessenberg-triangular (HT) reduction is performed, that is, orthogonal matrices $Q,Z \in \R^{n\times n}$ \change{are found} such that $H = Q^T A Z$ is in Hessenberg form (all entries below the sub-diagonal are zero) while $T = Q^T B Z$, \change{like $B$, is} in upper triangular form.
Third, $H$ is iteratively (and approximately) reduced further to quasi-triangular form, \change{allowing easier determination of} the eigenvalues of $A-\lambda B$ and associated quantities.

During the last decade, significant progress has been made \change{in speeding up} the third step, i.e., the iterative part of the QZ algorithm.
Its convergence has been accelerated by extending aggressive early deflation from the \change{QR algorithm~\cite{Braman2002a}} to the QZ algorithm~\cite{Kagstrom2006}.
Moreover, multi-shift techniques make sequential~\cite{Kagstrom2006} as well as parallel~\cite{Adlerborn2014} implementations perform well.

\change{As a} consequence of the improvements in the iterative part, the \change{HT reduction} of the matrix pencil has become \change{even more} critical to the performance of the QZ algorithm.
We mention in passing that this reduction also plays a role in aggressive early deflation and may thus become critical to the iterative part as well, at least in a parallel implementation~\cite{Adlerborn2014,Granat2010}.
The original algorithm for HT reduction from~\cite{Moler1973} reduces $A$ to Hessenberg form (and maintains $B$ in triangular form) by performing $\Theta(n^2)$ Givens rotations.
Even though progress has been made in~\cite{Kagstrom2008} to accumulate these Givens rotations and apply them more efficiently using matrix multiplication, the need for propagating sequences of rotations through the triangular matrix $B$ makes the sequential---but even more so the parallel---\change{implementations perform far below the peak rate of the machine.}

A general idea in dense eigenvalue solvers to speed up the \change{HT reduction} is to perform it in two (or more) stages.
For a single \emph{symmetric} matrix $A$, this idea amounts to reducing $A$ to banded form in the first stage and then further to tridiagonal form in the second stage.
Usually called successive band reduction~\cite{Bischof2000}, this currently appears to be the method of choice for tridiagonal reduction; see, e.g.,~\cite{Auckenthaler2011a,Bientinesi2011,Haidar2011,Haidar2013}.
However, this success story does not seem to carry over to the non-symmetric case, possibly because the second stage (reduction from block Hessenberg to Hessenberg form) is always an $\Omega(n^3)$ operation and hard to execute efficiently; see~\cite{Karlsson11a,Karlsson12} for some recent but limited progress.
The situation is certainly not simpler when reducing a matrix pencil $A -\lambda B$ to HT form~\cite{Kagstrom2008}. 

For the reduction of a single non-symmetric matrix to Hessenberg form, the classical Householder-based algorithm \cite{Dongarra1989,Quintana-Orti2006} remains the method of choice.
This is despite the fact that not all of its operations can be blocked, that is, a non-vanishing fraction of level~2 BLAS remains (approximately $20\%$ in the form of one matrix--vector multiplication \change{per column} involving the unreduced part).
Extending the use of (long) Householder reflectors (instead of Givens rotations) to HT reduction of a matrix pencil gives rise to a number of issues.
The aim of this paper is to describe how to satisfactorily address all of these issues.
We do so by combining an unconventional use of Householder reflectors with blocked updates of RQ decompositions.
We see the resulting Householder-based algorithm for HT reduction as a first step towards an algorithm that is more suitable for parallelization.
We provide some evidence in this direction, but the parallelization itself is out of scope and is deferred to future work.

The rest of this paper is organized as follows.
In Section~\ref{sec:prelims}, we recall the notions of (opposite) Householder reflectors and (compact) WY representations and their stability properties.
The new algorithm is described in Section~\ref{sec:algorithm} and numerical experiments are presented in Section~\ref{sec:experiments}.
The paper ends with conclusions and future work in Section~\ref{sec:conclusions}.

%!TEX root = 0__main.tex

\section{Preliminaries} \label{sec:prelims}

We recall the concepts of Householder reflectors \change{and the lesser known} opposite Householder reflectors, iterative refinement, and regular as well as compact WY representations.
These concepts are the main building blocks of the new algorithm.

\subsection{Householder reflectors}
\label{sec:conv-hous-refl}

We recall that an $n\times n$ Householder reflector takes the form
\[
 H = I - \beta v v^T, \qquad \beta = \frac{2}{v^T v}, \qquad v \in \R^n, \qquad \change{v\not=0,}
\]
where $I$ denotes the ($n\times n$) identity matrix.
Given a vector $x \in \R^n$, one can always choose \change{$v\not=0$} such that $Hx = \pm \|x\|_2 e_1$, with \change{$e_1$ being the first unit vector}; see~\cite[Sec.~5.1.2]{Golub2013} for details. 

Householder reflectors are orthogonal (and symmetric) and they represent one of the most common means to zero out entries in a matrix in a numerically stable fashion.
For example, by choosing $x$ to be the first column of an $n\times n$ matrix $A$, the application of $H$ from the left to $A$ reduces the first column of $A$, that is, the trailing $n-1$ entries in the first column of $HA$ are zero.

\subsection{Opposite Householder reflectors}
\label{sec:unconv-hous-refl}

What is less commonly known, and was possibly first noted in~\cite{Watkins2000}, is that Householder reflectors can be used in the opposite way, that is, a reflector can be applied \emph{from the right} to reduce a \emph{column} of a matrix.
To \change{illustrate the principle of constructing such opposite Householder reflectors}, let $B \in \R^{n\times n}$ be invertible and choose $x = B^{-1} e_1$.
Then the corresponding Householder reflector $H$ that reduces $x$ satisfies
\[
 (H B^{-1}) e_1 =\pm \| B^{-1} e_1\|_2 e_1 \qquad \Rightarrow \qquad (B H) e_1 = \pm \frac{1}{\| B^{-1} e_1\|_2} e_1.
\]
In other words, a reflector that reduces the first column of $B^{-1}$ \emph{from the left} (as in $H B^{-1}$) also reduces the first column of $B$ \emph{from the right} (as in $B H$).
\newchange{The following lemma is an extension of~\cite[Sec.~2.2]{Kagstrom2006}, and it provides an error analysis of opposite Householder reflectors. Note that the analysis allows for the inexact solution of $Bx = e_1$ and does not require $B$ to be invertible.}
\change{
\begin{lemma} \label{lemma:analysisopposite}
Let $\hat x \in \R^n$ satisfy $(B+\triangle) \hat x = e_1$ for matrices $B, \triangle \in \R^{n\times n}$. Consider the following procedure:
\begin{enumerate}
 \item Using~\cite[Alg.~5.1.1]{Golub2013}, compute coefficients $v,\beta$ of the Householder reflector $H = I - \beta v v^T$ such that $H\hat x = \pm \|\hat x\|_2 e_1$.
 \item Compute $C = B - (B v) (\beta v)^T$.
 \item Set entries $C_{2:n,1} \gets 0$.
\end{enumerate}
If this procedure is carried out in floating point arithmetic according to the standard model~\cite[Eq.~(2.4)]{Higham2002} then the computed output $\hat C$ satisfies
\[
 \hat C = (B+\hat \triangle )H, \qquad \|\hat \triangle\|_F \le \|\triangle\|_F + cn \mathrm u \|B\|_F,
\]
provided that $n \newchange{\mathrm u} \ll 1$, where $\mathrm u$ denotes the unit roundoff and $c$ is a small constant.
\end{lemma}
\begin{proof}
Let $\hat C_1$ denote the computed matrix $C$ after Step 2 of the procedure. By~\cite[Lemma 19.2]{Higham2002}, we have
\begin{equation} \label{eq:hatc0}
 \hat C_1 = (B+\hat \triangle_1) H, \qquad \|\hat \triangle_1\|_F \le \overline cn\mathrm u \|B\|_F,
\end{equation}
for some small constant $\overline c$. It remains to analyze Step 3. For this purpose, we first note that -- by assumption -- $\hat x\not=0$ and thus
$He_1 = \pm \frac{1}{\|\hat x\|_2} H^2 \hat x = \pm \frac{\hat x}{\|\hat x\|_2}$. Inserted into~\eqref{eq:hatc0}, this gives
\begin{eqnarray*}
  \pm \hat C_1 e_1 &=& \pm (B+\hat \triangle_1 )H e_1 =  (B+\hat \triangle_1 )  \frac{\hat x}{\|\hat x\|_2} \\
  &=& (B+\triangle )  \frac{\hat x}{\|\hat x\|_2} + (\hat \triangle_1 - \triangle) \frac{\hat x}{\|\hat x\|_2} 
  = \frac{e_1}{\|\hat x\|_2} + (\hat \triangle_1 - \triangle) \frac{\hat x}{\|\hat x\|_2}.
\end{eqnarray*}
Hence, setting the entries to zero below the diagonal in the first column of $\hat C_1$ corresponds to $\hat C = \hat C_1 + \hat \triangle_2$ with $\|\hat \triangle_2\|_F \le \|\hat \triangle_1\|_F+ \|\triangle\|_F$. Setting $\hat \triangle = \hat \triangle_1 + \hat \triangle_2 H$ and $c = 2 \bar c$ completes the proof.
\end{proof}
}

\change{Lemma~\ref{lemma:analysisopposite} shows that opposite Householder reflectors are numerically backward stable provided that $\|\triangle\|_F$ is not much larger than $\mathrm u \|B\|_{F}$ or, in other words, $Bx = e_1$ is solved in a backward stable manner.}
\begin{remark} \label{rem:singularB}
	In~\cite{Kagstrom2006}, \change{it was explained how the case of a singular matrix $B$ can be addressed by using an RQ decomposition of $B$. In our setting, such an approach is not feasible because the matrix $B$ is usually not explicitly available. Lemma~\ref{lemma:analysisopposite} suggests an alternative approach.}
	To define the Householder reflector for a singular matrix $B$, we replace it by a non-singular matrix \change{$B_0 = B + \triangle_0$} with a perturbation \newchange{$\triangle_0$} of norm $\mathcal O(\mathrm{u})\|B\|_F$. Assuming that \change{$B_0 x = e_1$} is solved in a \change{backward stable manner, the condition of Lemma~\ref{lemma:analysisopposite} is still met with $\|\triangle\|_F = \mathcal O(\mathrm{u})\|B\|_F$.} 
  \change{Below, in Section~\ref{sec:firstcol}, we discuss our specific choice of $\triangle_0$.}
\end{remark}

\subsection{Iterative refinement}
\label{sec:iterative-refinement}

The algorithm we are about to introduce operates in a setting for which the solver for $B x = e_1$ is \emph{not always} guaranteed to be stable.
We will therefore use iterative refinement (see, e.g.,~\change{\cite[Ch.~12]{Higham2002}}) to refine a computed solution $\hat x$:
\begin{enumerate}
 \item Compute the residual $r = e_1 - B \hat x$.
 \item Test convergence: Stop if \change{$\|r\|_2 / \|\hat x\|_2  \le 2 \mathrm{u} \|B\|_F$}.
 \item Solve correction equation $Bc = r$ (with unstable method).
 \item Update $\hat x \gets \hat x + c$ and repeat from Step 1.
\end{enumerate}
\change{By setting $\newchange{\triangle} = r \hat x^T / \|\hat x\|_2^2$, one observes that the condition of Lemma~\ref{lemma:analysisopposite} is satisfied with $\|\newchange{\triangle}\|_F = 2 \mathrm{u} \|B\|_F$ upon successful completion of iterative refinement.}

\subsection{Regular and compact WY representations} \label{sec:compactWY}

Let $I-\beta_i v_i v_i^T$ for $i = 1, 2, \ldots, k$ be Householder reflectors with $\beta_i \in \R$ and $\change{0 \neq v_i} \in \R^n$, \change{such that the first $(i-1)$ entries of $v_i$ are zero}.
Setting \[V = [v_1,\ldots,v_k] \in \R^{n\times k},\] there is an upper triangular matrix $T \in \R^{k\times k}$ such that
\begin{equation} \label{eq:compactWY}
  \prod_{i=1}^{k} (I-\beta_i v_i v_i^T) = I - V T V^T.
\end{equation}
This so-called \emph{compact WY representation}~\cite{Schreiber1989} allows for applying Householder reflectors in terms of matrix--matrix products (level 3 BLAS).
The LAPACK routines {\tt DLARFT} and {\tt DLARFB} can be used to construct and apply compact WY representation, respectively.

\change{When the number of reflectors $k$ is close to their length $n$,} the factor $T$ in~\eqref{eq:compactWY} constitutes a non-negligible contribution to the overall cost of applying \change{Householder reflectors in this representation}, \change{even more so when there is an additional zero pattern in the matrix $VT^T$ (e.g., the last $(k-i)$ entries in $v_i$ are zero)}.
In these cases, we instead use a \emph{regular WY representation}~\cite[Method~2]{Bischof1987}, which takes the form $I - V W^{T}$ with $W = V T^{T}$.

%!TEX root = 0__main.tex

\section{Algorithm} \label{sec:algorithm}

Throughout this section, which is devoted to the description of the new algorithm, we assume that $B$ has already been reduced to triangular form, e.g., by an RQ decomposition.
For simplicity, we will also assume that $B$ is non-singular (see Remark~\ref{rem:singularB} for how to eliminate this assumption). \change{The matrices $Q$ and $Z$, which will accumulate orthogonal transformations, are initialized to identity.}

\subsection{Overview} \label{sec:overview}

We first introduce the basic idea of the algorithm before going through most of the details.

The algorithm proceeds as follows.
The first column of $A$ is reduced below the first sub-diagonal by a conventional reflector from the left.
When this reflector is applied from the left to $B$, every column except the first fills in:
\[
 (A,B) \gets \small \left( \begin{bmatrix}
\blackx & \blackx &\blackx &\blackx &\blackx \\
\redx & \redx &\redx &\redx &\redx  \\
\blueo & \redx &\redx &\redx &\redx  \\
\blueo & \redx &\redx &\redx &\redx  \\
\blueo & \redx &\redx &\redx &\redx  
 \end{bmatrix}, \begin{bmatrix}
\blackx & \blackx &\blackx &\blackx &\blackx \\
\blacko & \redx &\redx &\redx &\redx  \\
\blacko & \redx &\redx &\redx &\redx  \\
\blacko & \redx &\redx &\redx &\redx  \\
\blacko & \redx &\redx &\redx &\redx  
 \end{bmatrix}\right).
\]
The second column of $B$ is reduced below the diagonal by an opposite reflector \emph{from the right}, as described in Section~\ref{sec:unconv-hous-refl}.
Note that the computation of this reflector requires the (stable) solution of a linear system involving the matrix $B$.
When the reflector is applied from the right to $A$, its first column is preserved:
\[
 (A,B) \gets \small \left( \begin{bmatrix}
\blackx & \redx &\redx &\redx &\redx \\
\blackx & \redx &\redx &\redx &\redx  \\
\blacko & \redx &\redx &\redx &\redx  \\
\blacko & \redx &\redx &\redx &\redx  \\
\blacko & \redx &\redx &\redx &\redx  
 \end{bmatrix}, \begin{bmatrix}
\blackx & \redx &\redx &\redx &\redx \\
\blacko & \redx &\redx &\redx &\redx  \\
\blacko & \blueo &\redx &\redx &\redx  \\
\blacko & \blueo &\redx &\redx &\redx  \\
\blacko & \blueo &\redx &\redx &\redx  
 \end{bmatrix}\right).
\]
Clearly, the idea can be repeated for the \emph{second} column of $A$ and the \emph{third} column of $B$, and so on:
\[
\small \left( \begin{bmatrix}
\blackx & \blackx  &\redx &\redx &\redx \\
\blackx & \blackx  &\redx &\redx &\redx  \\
\blacko & \redx &\redx &\redx &\redx  \\
\blacko & \blueo &\redx &\redx &\redx  \\
\blacko & \blueo &\redx &\redx &\redx  
 \end{bmatrix}, \begin{bmatrix}
\blackx & \blackx &\redx &\redx &\redx \\
\blacko & \blackx &\redx &\redx &\redx  \\
\blacko & \blacko &\redx &\redx &\redx  \\
\blacko & \blacko &\blueo &\redx &\redx  \\
\blacko & \blacko &\blueo &\redx &\redx  
 \end{bmatrix}\right), \ \ \small \left( \begin{bmatrix}
\blackx & \blackx  &\blackx &\redx &\redx \\
\blackx & \blackx  &\blackx &\redx &\redx  \\
\blacko & \blackx &\blackx &\redx &\redx  \\
\blacko & \blacko &\redx &\redx &\redx  \\
\blacko & \blacko &\blueo &\redx &\redx  
 \end{bmatrix}, \begin{bmatrix}
\blackx & \blackx &\blackx &\redx &\redx \\
\blacko & \blackx &\blackx &\redx &\redx  \\
\blacko & \blacko &\blackx &\redx &\redx  \\
\blacko & \blacko &\blacko &\redx &\redx  \\
\blacko & \blacko &\blacko &\blueo &\redx  
 \end{bmatrix}\right).
\]
After a total of $n - 2$ steps, the matrix $A$ will be in upper Hessenberg form and $B$ will be in upper triangular form, i.e., the reduction to Hessenberg-triangular form will be complete.
This is the gist of the new algorithm. 
The reduction is carried out by $n - 2$ conventional reflectors applied from the left to reduce columns of $A$ and $n - 2$ opposite reflectors applied from the right to reduce columns of $B$.

A naive implementation of the algorithm sketched above would require as many as $\Omega (n^{4})$ operations simply because each of the $n - 2$ iterations requires the solution of a dense linear system with the unreduced part of $B$, whose size is roughly $n / 2$ on average.
In addition to this unfavorable complexity, the arithmetic intensity \change{(i.e., the flop-to-memory-reference ratio)} of the $\Theta (n^{3})$ flops associated with the application of individual reflectors will be very low.
The following two ingredients aim at addressing both of these issues:
\begin{enumerate}
\item
  The arithmetic intensity is increased for a majority of the flops associated with the application of reflectors by performing the reduction in \emph{panels} (i.e., a small number of consecutive columns), delaying some of the updates, and using compact WY representations.
  The details resemble the blocked algorithm for Hessenberg reduction~\cite{Dongarra1989,Quintana-Orti2006}.
  
\item
  To reduce the complexity from $\Theta (n^{4})$ to $\Theta (n^{3})$, we avoid applying reflectors directly to $B$.
  Instead, we keep $B$ in factored form during the reduction of a panel:
  \begin{equation}
    \label{eq:tildeB}
    \tilde B = ( I - U S U^{T} )^{T} B ( I - V T V^{T} ).  
  \end{equation}
  Since $B$ is triangular and the other factors are orthogonal, this reduces the cost for solving a system of equations with $\tilde B$ from $\Theta(n^{3})$ to $\Theta (n^{2})$.
  For reasons explained in Section~\ref{sec:columnj} below, this approach is \emph{not always} numerically backward stable.
  A fall-back mechanism is therefore necessary to guarantee stability; \change{the mechanism used in the new algorithm is described in the following sections. Moreover, iterative refinement is used to avoid triggering the fall-back mechanism in many cases. Numerical experiments show that the combination of iterative refinement and the fall-back mechanism typically only slightly degrades the performance of the algorithm, while keeping it provably stable.}
  % The new algorithm uses a fall-back mechanism that only slightly degrades the performance.
  % Moreover, iterative refinement is used to avoid triggering the fall-back mechanism in many cases. 
  After the reduction of a panel is completed, $\tilde B$ is returned to upper triangular form in an efficient manner.
\end{enumerate}

\subsection{Panel reduction}

Let us suppose that the first $s-1$ (with $0\le s-1 \le n-3$) columns of $A$ have already been reduced (and hence $s$ is the first unreduced column) and $B$ is in upper triangular form (i.e., not in factored form \eqref{eq:tildeB}).
The matrices $A$ and $B$ take the shapes depicted in Figure~\ref{fig:matrix-shapes} for $j = s$.
In the following, we describe a reflector-based algorithm that aims at reducing the panel containing the next $\nb$ unreduced columns of $A$.
The algorithmic parameter $\nb$ should be tuned to maximize performance (see also Section~\ref{sec:experiments} for the choice of $\nb$).

\begin{figure}[htb]
  \centering
  \begin{tikzpicture}
    [scale=0.4,
    y=-1cm]
    
    \begin{scope}
      \node [above] at (0.5,0) {$U$};
      \node [below] at (0.5,10) {$k$};
      \draw [|-|] (-2,0) -- (-2,10) node [midway,rotate=90,left,anchor=south] {$n$};
      \draw [|-|] (-0.5,0) -- (-0.5,2) node [midway,left,rotate=90,anchor=south] {$s$};
      \draw [|-|] (-0.5,2) -- (-0.5,10) node [midway,left,rotate=90,anchor=south] {$n - s$};
      \draw [] (0,2) -- (0,0) -- (1,0) -- (1,3);
      \draw [fill=red!20] (0,2) -- (1,3) -- (1,10) -- (0,10) -- cycle;

      \begin{scope}
        [xshift=2cm,yshift=-4.5cm]
        
        \node [above] at (0.5,0) {$S$};
        \node [below] at (0.5,1) {$k$};
        \draw [fill=red!20] (0,0) -- (1,0) -- (1,1) -- cycle;
      \end{scope}
    \end{scope}
    
    \begin{scope}[xshift=6cm]
      \node [above] at (0.5,0) {$V$};
      \node [below] at (0.5,10) {$k$};
      \draw [|-|] (-0.5,0) -- (-0.5,2) node [midway,left,rotate=90,anchor=south] {$s$};
      \draw [|-|] (-0.5,2) -- (-0.5,10) node [midway,left,rotate=90,anchor=south] {$n - s$};
      \draw [] (0,2) -- (0,0) -- (1,0) -- (1,3);
      \draw [fill=red!20] (0,2) -- (1,3) -- (1,10) -- (0,10) -- cycle;

      \begin{scope}
        [xshift=2cm,yshift=-4.5cm]
        
        \node [above] at (0.5,0) {$T$};
        \node [below] at (0.5,1) {$k$};
        \draw [fill=red!20] (0,0) -- (1,0) -- (1,1) -- cycle;
      \end{scope}
    \end{scope}

    \begin{scope}[xshift=12cm]
      \node [above] at (5,0) {$A$};
      \draw [|-|] (0,10.5) -- (3,10.5) node [midway,below] {$j - 1$};
      \draw [|-|] (3,10.5) -- (10,10.5) node [midway,below] {$n - j + 1$};
      \draw [fill=blue!20] (0,0) -- (10,0) -- (10,10) -- (3,10) -- (3,3) -- cycle;
      \draw [] (0,0.2) -- (3,3.2);
      \draw [dashed] (1.8,0) -- (1.8,8);
      \draw [|-|] (0,8) -- (1.8,8) node [midway,below] {\small $s - 1$};
      \draw [|-|] (1.8,8) -- (3,8) node [near start,below] {\small $k$};
    \end{scope}

    \begin{scope}
      [xshift=-2cm,yshift=-14cm]

      \node [above] at (5,0) {$B$};
      \draw [fill=green!20] (0,0) -- (10,0) -- (10,10) -- cycle;
    \end{scope}
    
    \begin{scope}
      [xshift=12cm,yshift=-14cm]

      \node [above] at (5,0) {$\tilde B = (I-USU^T)^TB(I-VTV^T)$};
      \draw [|-|] (0,10.5) -- (3.3,10.5) node [midway,below] {$j$};
      \draw [|-|] (3.3,10.5) -- (10,10.5) node [midway,below] {$n - j$};
      \draw [fill=green!20] (0,0) -- (10,0) -- (10,10) -- (3.3,10) -- (3.3,3.3) -- cycle;
    \end{scope}

  \end{tikzpicture}
  \caption{Illustration of the shapes and sizes of the matrices involved in the reduction of a panel at the \emph{beginning} of the $j$th step of the algorithm, where $j \in [s, s + \nb)$.}
  \label{fig:matrix-shapes}
\end{figure}
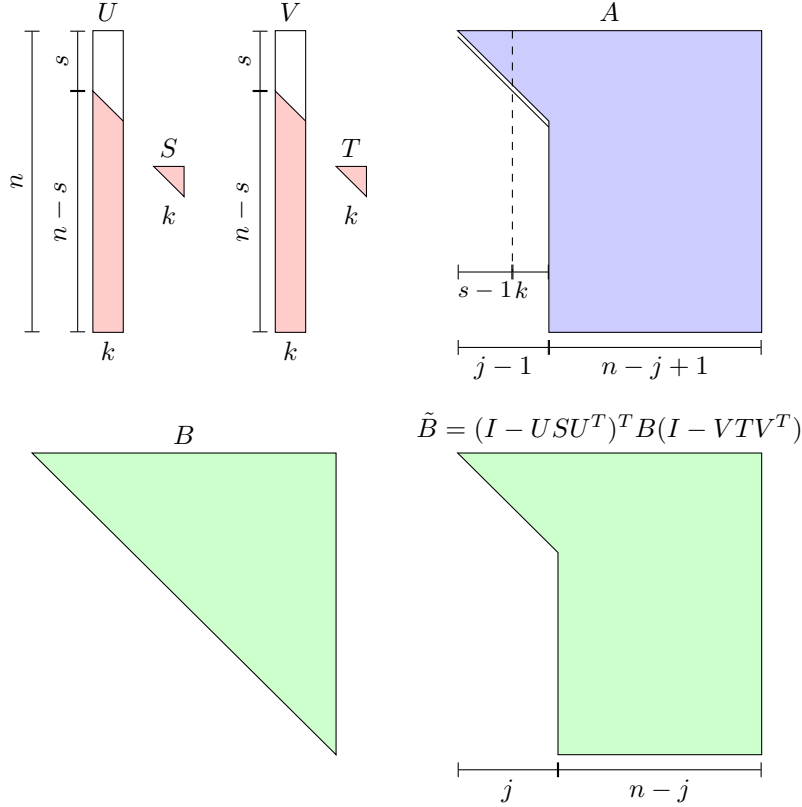

\subsubsection{Reduction of the first column ($j = s$) of a panel} \label{sec:firstcol}

In the first step of a panel reduction, a reflector $I - \beta uu^T$ is constructed to reduce column $j=s$ of $A$.
Except for entries in this particular column, no other entries of $A$ are updated at this point.
Note that the first $j$ entries of $u$ are zero and hence the first $j$ columns of $\tilde B = ( I - \beta u u^T ) B$ will remain in upper triangular form.
Now to reduce column $j+1$ of $\tilde B$, we need to solve, according to Section~\ref{sec:unconv-hous-refl}, the linear system
\[
 \tilde B_{j+1:n, j+1:n}x = \left( I - \beta u_{j+1:n} u_{j+1:n}^T \right) B_{j+1:n,j+1:n} x = e_1.
\]
The solution vector is given by
\begin{displaymath}
  x =
  B_{j+1:n,j+1:n}^{-1} \left( I - \beta u_{j+1:n} u_{j+1:n}^T \right) e_1 =
  B_{j+1:n,j+1:n}^{-1} \underbrace{\left( e_{1} - \beta u_{j+1:n} u_{j+1} \right)}_{y}.
\end{displaymath}
In other words, we first form the dense vector $y$ and then solve an upper triangular linear system with $y$ as the right-hand side. 
\change{If the matrix $B_{j+1:n,j+1:n}$ contains a zero on the diagonal, we replace the zero with $2\newchange{\mathrm u}\rho \|B\|_F$, where $\rho$ is a normally distributed random number. This way, as explained in Remark~\ref{rem:singularB},} both the formation of $y$ and the solution of the triangular system are backward stable~\change{\cite{Higham2002}} and hence \change{the condition of Lemma~\ref{lemma:analysisopposite} is satisified. In turn,}
the resulting Householder reflector $( I - \gamma v v^T )$ reliably yields a reduced $(j+1)$th column in $( I - \beta u u^T ) B ( I - \gamma v v^T )$.
We complete the reduction of the first column of the panel by initializing
\[
 U \gets u, \quad S \gets [\beta], \quad V \gets v, \quad T \gets [\gamma], \quad Y \gets \change{\gamma} Av.
\]
\change{The role of the matrix $Y$ will be to hold the product $Y=AVT$, similarly as in blocked algorithms for Hessenberg reduction~\cite{Dongarra1989,Quintana-Orti2006}; see also the LAPACK routine DGEHRD.}

\begin{remark} \label{rem:comp-Y}
  For simplicity, we assume that all rows of $Y$ are computed during the panel reduction.
  In practice, the first few rows of $Y = A V T$ are computed later on in a more efficient manner as described in~\cite{Quintana-Orti2006}.
\end{remark}

\subsubsection{Reduction of subsequent columns ($j>s$) of a panel} \label{sec:columnj}

We now describe the reduction of column $j \in (s, s + \nb)$, assuming that the previous $k = j - s \geq 1$ columns of the panel have already been reduced.
This situation is illustrated in Figure~\ref{fig:matrix-shapes}. 
At this point, $I-USU^T$ and $I-VTV^T$ are the compact WY representations of the $k$ previous reflectors from the left and the right, respectively.
The transformed matrix $\tilde B$ is available only in the factored form~\eqref{eq:tildeB}, with the upper triangular matrix $B$ remaining unmodified throughout the entire panel reduction.
Similarly, most of $A$ remains unmodified except for the reduced part of the panel.

\begin{paragraph}{a) Update column $j$ of $A$}
  To prepare its reduction, the $j$th column of $A$ is updated with respect to the $k$ previous reflectors:
  \begin{eqnarray*}
    A_{:,j} &\gets& A_{:,j} - Y V_{j,:}^{T}, \\
    A_{:,j} &\gets& A_{:,j} - U S^{T} U^{T} A_{:,j}.
  \end{eqnarray*}
  Note that \change{because of} Remark~\ref{rem:comp-Y}, actually only rows $s + 1 : n$ of $A$ need to be updated at this point.
\end{paragraph}

\begin{paragraph}{b) Reduce column $j$ of $A$ from the left}
  Construct a reflector $I - \beta u u^{T}$ such that it reduces the $j$th column of $A$ below the first sub-diagonal:
  \[
    A_{:,j} \gets (I - \beta u u^{T} ) A_{:,j}.
  \]
  The new reflector is absorbed into the compact WY representation by
  \begin{equation*}
%    \label{eq:extend-US}
    U \gets \begin{bmatrix} U & u \end{bmatrix}, \qquad S \gets \begin{bmatrix} S & - \beta S U^{T} u \\ 0 & \beta \end{bmatrix}.
  \end{equation*}
  %See also Section~\ref{sec:compactWY}.
\end{paragraph}

\begin{paragraph}{c) Attempt to solve a linear system in order to reduce column $j + 1$ of $\tilde B$}
  This step aims at (implicitly) reducing the $(j+1)$th column of $\tilde B$ defined in~\eqref{eq:tildeB} by an opposite reflector from the right.
  As illustrated in Figure~\ref{fig:matrix-shapes}, $\tilde B$ is block upper triangular:
  \begin{displaymath}
    \tilde B =
    \begin{bmatrix}
      \tilde B_{11} & \tilde B_{12} \\
      0             & \tilde B_{22} \\
    \end{bmatrix}, \qquad \tilde B_{11} \in \R^{j\times j}, \qquad \tilde B_{22} \in \R^{(n-j)\times (n-j)}.
  \end{displaymath}
  To simplify the notation, the following description uses the full matrix $\tilde B$ whereas in practice we only need to work with the sub-matrix that is relevant for the reduction of the current panel, namely, $\tilde B_{s+1:n, s+1:n}$.

  According to Section~\ref{sec:unconv-hous-refl}, we need to solve the linear system
  \begin{equation}
    \label{eq:tildeB22x-eq-c}
    \tilde B_{22}x = c, \qquad c = e_{1}
  \end{equation}
  in order to determine an opposite reflector from the right that reduces the first column of $\tilde B_{22}$.
  However, because of the factored form~\eqref{eq:tildeB}, we do not have direct access to $\tilde B_{22}$ and we therefore instead work with the enlarged system
  \begin{equation} \label{eq:enlargedsystem}
    \tilde B y = 
    \begin{bmatrix}
      \tilde B_{11} & \tilde B_{12} \\
      0             & \tilde B_{22} \\
    \end{bmatrix}
    \begin{bmatrix}
      y_{1} \\
      y_{2} \\
    \end{bmatrix}
    =
    \begin{bmatrix}
      0 \\
      c \\
    \end{bmatrix}.
  \end{equation}
  From the enlarged solution vector $y$ we can extract the desired solution vector $x = y_{2} = \tilde B_{22}^{-1} c$. 
  By combining~\eqref{eq:tildeB} and the orthogonality of the factors with (\ref{eq:enlargedsystem}) we obtain
  \begin{displaymath}
    x = E^{T}( I - V T V^{T} )^{T} B^{-1} ( I - U S U^{T} )   \begin{bmatrix}
      0 \\
      c \\
    \end{bmatrix}, \quad \text{with} \quad E = 
    \begin{bmatrix}
      0 \\
      I_{n-j} \\
    \end{bmatrix}.
  \end{displaymath}
  We are lead to the following procedure for solving (\ref{eq:tildeB22x-eq-c}):
  \begin{enumerate}
  \item Compute $\tilde c \gets (I - U S U^{T}) \begin{bmatrix} 0 \\ c \end{bmatrix}$.
  \item Solve the triangular system $B \tilde y = \tilde c$ by backward substitution.
  \item Compute the enlarged solution vector $y \gets (I - V T V^{T})^{T} \tilde y$.
  \item Extract the desired solution vector $x \gets y_{j+1:n}$.
  \end{enumerate}
  While only requiring $\Theta(n^2)$ operations, this procedure is in general \emph{not} backward stable for $j > s$.
  When $\tilde B$ is significantly more ill-conditioned than $\tilde B_{22}$ alone, the intermediate vector $y$ (or, equivalently, $\tilde y$) may have a much larger norm than the desired solution vector $x$ leading to subtractive cancellation in the third step.
  As HT reduction has a tendency to move tiny entries on the diagonal of $B$ to the top left corner~\cite{Watkins2000}, we expect this instability to be more prevalent during the reduction of the first few panels (and this is indeed what we observe in the experiments in Section~\ref{sec:experiments}).

  To test backward stability of a computed solution $\hat x$ of (\ref{eq:tildeB22x-eq-c}) and perform iterative refinement, if needed, we compute the residual $r = c - \tilde B_{22} \hat x$ as follows:
  \begin{enumerate}
  \item Compute $w \gets (I - V T V^{T}) \begin{bmatrix} 0 \\ \hat x \end{bmatrix}$.
  \item Compute $w \gets B w$.
  \item Compute $w \gets (I - U S^{T} U^{T}) w$.
  \item Compute $r \gets c - w_{j+1:n}$.
  \end{enumerate}
  We perform the iterative refinement procedure described in Section~\ref{sec:iterative-refinement} as long as $\|r\|_2 \change{/ \|\hat x\|_{2}} > \mathsf{tol} = 2 \mathrm{u} \|B\|_F$ but abort after ten iterations.
  In the rare case when this procedure does not converge, we prematurely stop the current panel reduction and absorb the current set of reflectors as described in Section~\ref{sec:absorb} below.
  We then start over with a new panel reduction starting at column $j$.
  It is important to note that the algorithm is now \emph{guaranteed} to make progress since when $k = 0$ we have $\tilde B = B$ and therefore solving (\ref{eq:tildeB22x-eq-c}) is backward stable.
\end{paragraph}

\begin{paragraph}{d) Implicitly reduce column $j + 1$ of $\tilde B$ from the right}

  Assuming that the previous step computed an accurate solution vector $x$ to (\ref{eq:tildeB22x-eq-c}), we can continue with this step to complete the implicit reduction of column $j + 1$ of $\tilde B$.
  If the previous step failed, then we simply skip this step.
  A reflector $I - \gamma v v^{T}$ that reduces $x$ is constructed and absorbed into the compact WY representation as in
  \begin{displaymath}
    V \gets \begin{bmatrix} V & v \end{bmatrix}, \qquad T \gets \begin{bmatrix} T & - \gamma T V^{T} v \\ 0 & \gamma \end{bmatrix}.
  \end{displaymath}
  At the same time, a new column $y$ is appended to $Y$:
  \begin{displaymath}
    y \gets \gamma ( A v - Y V^{T} v ), \qquad Y \gets \begin{bmatrix} Y & y \end{bmatrix}.
  \end{displaymath}
  Note the common sub-expression $V^{T} v$ in the updates of $T$ and $Y$.
  Following Remark~\ref{rem:comp-Y}, the first $s$ rows of $Y$ are computed later in practice.
\end{paragraph}

\subsection{Absorption of reflectors}
\label{sec:absorb}

The panel reduction normally terminates after $k = \nb$ steps.
In the rare event that iterative refinement fails, the panel reduction will terminate prematurely after only $k \in [1, \nb)$ steps.
Let $k \in [1, \nb]$ denote the number of left and right reflectors accumulated during the panel reduction.
The aim of this section is to describe how the $k$ left and right reflectors are absorbed into $A$, $B$, $Q$, and $Z$ so that the next panel reduction is ready to start with $s \gets s + k$.

We recall that Figure~\ref{fig:matrix-shapes} illustrates the shapes of the matrices at this point.
The following facts are central:
\begin{enumerate}[{Fact} 1.]
\item \label{fact2} Reflector $i = 1, 2, \ldots, k$ affects entries $s+i:n$. In particular, entries $1:s$ are unaffected.
\item \label{fact4} The first $j - 1$ columns of $A$ have been updated and their rows $j+1:n$ are zero.
\item \label{fact5} The matrix $\tilde B$ is in upper triangular form in its first $j$ columns.
\end{enumerate}

In principle, it would be straightforward to apply the left reflectors to $A$ and $Q$ and the right reflectors to $A$ and $Z$.
The only complications arise from the need to preserve the triangular structure of $B$. 
To update $B$ one would need to perform a transformation of the form
\begin{equation} \label{eq:updateB}
  B \gets ( I - U S U^{T} )^{T} B ( I - V T V^{T} ).
\end{equation}
However, once this update is executed, the restoration of the triangular form (e.g., by an RQ decomposition) of \change{the now fully populated matrix} $B$  would have $\Theta(n^3)$ complexity, leading to an overall complexity of $\Theta(n^4)$.
In order to keep the complexity down, a very different approach is pursued:
\change{we will additionally transform $U$ and $V$ in a series of steps. Each step will introduce more zeros into $U$ and $V$, while making $B$ only block upper Hessenberg instead of full. This way, the final restoration of the triangular form of $B$ is significantly cheaper.}
% This entails additional transformations of both $U$ and $V$ that considerably increase their sparsity.
In the following, we use the term \emph{absorption} (instead of updating) to emphasize the presence of these additional transformations, which affect $A$, $Q$, and $Z$ as well.
\change{To simplify the exposition, we assume that 
% $k \,|\, (n-j)$, 
\newchange{$k$ divides $n-j$}, 
and comment on the general case in Remark~\ref{remark:notdivide} below.}

\subsubsection{Absorption of right reflectors} \label{sec:absorbright}

The aim of this section is to show how the right reflectors $I - V T V^{T}$ are absorbed into $A$, $B$, and $Z$ while (nearly) preserving the upper triangular structure of $B$.
When doing so we restrict ourselves to adding transformations only from the right due to the need to preserve the structure of the pending left reflectors, \change{i.e., $U$ and $S$ in~\eqref{eq:updateB} are not modified}.

\begin{paragraph}{a) Initial situation} We partition $V$ as 
$
  V = \small
  \begin{bmatrix}
    0 \\
    V_{1} \\
    V_{2} \\
  \end{bmatrix},
$
where $V_{1}$ is a lower triangular $k \times k$ matrix starting at row~$s + 1$ (Fact~\ref{fact2}).
Hence $V_{2}$ starts at row~$j + 1$ (recall that $k = j - s$).
Our initial aim is to absorb the update
\begin{equation} \label{eq:initright}
B \gets B(I - V T V^{T}) = B  \left( I -   \begin{bmatrix}
    0 \\
    V_{1} \\
    V_{2} \\
  \end{bmatrix} T
    \begin{bmatrix}
    0 &
    V^T_{1} &
    V^T_{2} 
  \end{bmatrix}
\right).
\end{equation}
The shapes of $B$ and $V$ are illustrated in Figure~\ref{fig:right-incorporate}~(a).
\end{paragraph}

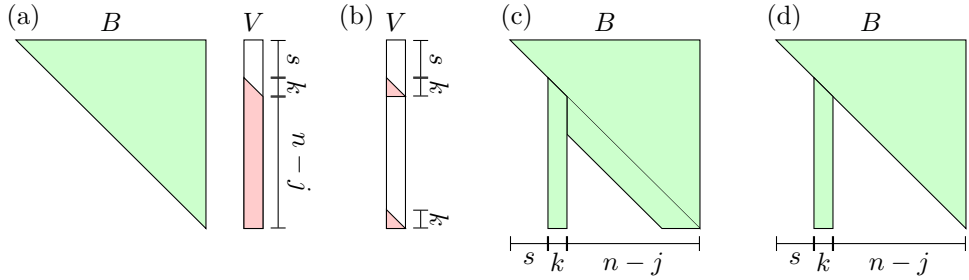
\begin{figure}[htb]
  \centering
  \begin{tikzpicture}
    [scale=0.25,y=-1cm]
    
    \begin{scope}
      \node [above right] at (-13,0) {(a)};
      \node [above] at (-7,0) {$B$};
      \node [above] at (0.5,0) {$V$};
      % B
      \fill [green!20] (-12,0) -- (-2,0) -- (-2,10) -- cycle;
      \draw (-12,0) -- (-2,0) -- (-2,10) -- cycle;

      \draw (0,0) rectangle (1,10);
      \fill [red!20] (0,2) -- (1,3) -- (1,10) -- (0,10) -- cycle;
      \draw (0,2) -- (1,3) -- (1,10) -- (0,10) -- cycle;
      \draw [|-|] (1.8,0) -- (1.8,2) node [midway,right,rotate=-90,anchor=south] {$s$};
      \draw [|-|] (1.8,2) -- (1.8,3) node [midway,right,rotate=-90,anchor=south] {$k$};
      \draw [|-|] (1.8,3) -- (1.8,10) node [midway,right,rotate=-90,anchor=south] {$n - j$};
    \end{scope}
    
     \begin{scope}[yshift=0cm, xshift=7.5cm]
      \node [above right] at (-3,0) {(b)};
      \node [above] at (0.5,0) {$V$};
      
      % V
      %\fill [red!20] (0,0) rectangle (1,10);
      \draw (0,0) rectangle (1,10);
      \fill [red!20] (0,2) -- (1,3) -- (0,3) -- cycle;
      \draw [|-|] (0,2) -- (1,3) -- (0,3) -- cycle;
      \fill [red!20] (0,9) -- (1,10) -- (0,10) -- cycle;
      \draw [|-|] (0,9) -- (1,10) -- (0,10) -- cycle;
      \draw [|-|] (1.8,0) -- (1.8,2) node [midway,right,rotate=-90,anchor=south] {$s$};
      \draw [|-|] (1.8,2) -- (1.8,3) node [midway,right,rotate=-90,anchor=south] {$k$};
      \draw [|-|] (1.8,9) -- (1.8,10) node [midway,right,rotate=-90,anchor=south] {$k$};
    \end{scope}

    \begin{scope}
      [yshift=0cm, xshift=26cm]
      \node [above right] at (-13,0) {(c)};
      \node [above] at (-7,0) {$B$};

      % B
      \fill [green!20] (-12,0) -- (-2,0) -- (-2,10) -- cycle;
      \draw (-12,0) -- (-2,0) -- (-2,10) -- cycle;
      \fill [green!20] (-9,3) -- (-9,5) -- (-4,10) -- (-2,10) -- cycle;
      \draw (-9,3) -- (-9,5) -- (-4,10) -- (-2,10);
      \fill [green!20] (-10,2) -- (-10,10) -- (-9,10) -- (-9,3) -- cycle;
      \draw (-10,2) -- (-10,10) -- (-9,10) -- (-9,3) -- cycle;

      \draw [|-|] (-12,10.8) -- (-10,10.8) node [midway,below] {$s$};
      \draw [|-|] (-10,10.8) -- (-9,10.8) node [midway,below] {$k$};
      \draw [|-|] (-9,10.8) -- (-2,10.8) node [midway,below] {$n - j$};

    \end{scope}

    \begin{scope}
      [yshift=0cm, xshift=40cm]
      \node [above right] at (-13,0) {(d)};
      \node [above] at (-7,0) {$B$};

      % B
      \fill [green!20] (-12,0) -- (-2,0) -- (-2,10) -- cycle;
      \draw (-12,0) -- (-2,0) -- (-2,10) -- cycle;
      \fill [green!20] (-10,2) -- (-10,10) -- (-9,10) -- (-9,3) -- cycle;
      \draw (-10,2) -- (-10,10) -- (-9,10) -- (-9,3) -- cycle;
      \draw [|-|] (-12,10.8) -- (-10,10.8) node [midway,below] {$s$};
      \draw [|-|] (-10,10.8) -- (-9,10.8) node [midway,below] {$k$};
      \draw [|-|] (-9,10.8) -- (-2,10.8) node [midway,below] {$n - j$};
    \end{scope}
  \end{tikzpicture}
  \caption{Illustration of the shapes of $B$ and $V$ when absorbing right reflectors into $B$: (a) initial situation, (b) after reduction of $V$, (c) after applying orthogonal transformations to $B$, (d) after partially restoring $B$.}
  \label{fig:right-incorporate}
\end{figure}

\begin{paragraph}{b) Reduce $V$} 

  We reduce the $(n-j) \times k$ matrix $V_2$ to lower triangular from via a sequence of QL decompositions from top to bottom.
  For this purpose, a QL decomposition of rows $1,\ldots,2k$ is computed, then a QL decomposition of rows $k+1,\ldots,3k$, etc.
  After a total of $r \change{ = } (n-j-k)/k$ such steps,  we arrive at the desired form:
\[
\left[\arraycolsep=1.4pt\renewcommand{\arraystretch}{0.6} 
\small \begin{array}{ccc}
\blackx & \blackx & \blackx \\
\blackx & \blackx & \blackx \\
\blackx & \blackx & \blackx \\ \hline  \\[-6pt]
\blackx & \blackx & \blackx \\
\blackx & \blackx & \blackx \\
\blackx & \blackx & \blackx \\ \hline  \\[-6pt]
\blackx & \blackx & \blackx \\
\blackx & \blackx & \blackx \\
\blackx & \blackx & \blackx \\ \hline  \\[-6pt]
\blackx & \blackx & \blackx \\
\blackx & \blackx & \blackx \\
\blackx & \blackx & \blackx \\ \hline  \\[-6pt]
\blackx & \blackx & \blackx \\
\blackx & \blackx & \blackx \\
\blackx & \blackx & \blackx 
\end{array}
\right] \stackrel{\hat Q_1}{\longrightarrow}
\left[\arraycolsep=1.4pt\renewcommand{\arraystretch}{0.6} 
\small \begin{array}{ccc}
\blueo & \blueo & \blueo \\
\blueo & \blueo & \blueo \\
\blueo & \blueo & \blueo \\ \hline  \\[-6pt]
\redx & \blueo & \blueo \\
\redx & \redx  & \blueo \\
\redx & \redx & \redx \\ \hline  \\[-6pt]
\blackx & \blackx & \blackx \\
\blackx & \blackx & \blackx \\
\blackx & \blackx & \blackx \\ \hline  \\[-6pt]
\blackx & \blackx & \blackx \\
\blackx & \blackx & \blackx \\
\blackx & \blackx & \blackx \\ \hline  \\[-6pt]
\blackx & \blackx & \blackx \\
\blackx & \blackx & \blackx \\
\blackx & \blackx & \blackx 
\end{array}
\right] \stackrel{\hat Q_2}{\longrightarrow}
\left[\arraycolsep=1.4pt\renewcommand{\arraystretch}{0.6} 
\small \begin{array}{ccc}
\blacko & \blacko & \blacko \\
\blacko & \blacko & \blacko \\
\blacko & \blacko & \blacko \\ \hline  \\[-6pt]
\blueo & \blacko & \blacko \\
\blueo & \blueo & \blacko \\
\blueo & \blueo & \blueo \\ \hline  \\[-6pt]
\redx & \blueo & \blueo \\
\redx & \redx  & \blueo \\
\redx & \redx & \redx \\ \hline  \\[-6pt]
\blackx & \blackx & \blackx \\
\blackx & \blackx & \blackx \\
\blackx & \blackx & \blackx \\ \hline  \\[-6pt]
\blackx & \blackx & \blackx \\
\blackx & \blackx & \blackx \\
\blackx & \blackx & \blackx 
\end{array}
\right] \quad \cdots \quad \stackrel{\hat Q_r}{\longrightarrow}
\left[\arraycolsep=1.4pt\renewcommand{\arraystretch}{0.6} 
\small \begin{array}{ccc}
\blacko & \blacko & \blacko \\
\blacko & \blacko & \blacko \\
\blacko & \blacko & \blacko \\ \hline  \\[-6pt]
\blacko & \blacko & \blacko \\
\blacko & \blacko & \blacko \\
\blacko & \blacko & \blacko \\ \hline  \\[-6pt]
\blacko & \blacko & \blacko \\
\blacko & \blacko & \blacko \\
\blacko & \blacko & \blacko \\ \hline  \\[-6pt]
\blueo & \blacko & \blacko \\
\blueo & \blueo & \blacko \\
\blueo & \blueo & \blueo \\ \hline  \\[-6pt]
\redx & \blueo & \blueo \\
\redx & \redx  & \blueo \\
\redx & \redx & \redx 
\end{array}
\right].
\]
This corresponds to a decomposition of the form
\begin{equation} \label{eq:lqdecomposition}
V_{2} = \hat Q_{1} \cdots \hat Q_{r} \hat L \quad \text{with} \quad \hat L = \begin{bmatrix} 0 \\ \hat L_1 \end{bmatrix},
\end{equation}
where each factor $\hat Q_j$ has a regular WY representation of size at most $2k\times k$ and $\hat L_1$ is a lower triangular $k\times k$ matrix.
%The triangular structure of this regular WY representation is exploited when applying $\hat Q_j$.
\end{paragraph}

\begin{paragraph}{c) Apply orthogonal transformations to $B$} 
After multiplying~\eqref{eq:initright} with $\hat Q_{1} \cdots \hat Q_{r}$ from the right, we get 
\begin{eqnarray}
 B &\gets& B\left( I - \begin{bmatrix}
    0 \\
    V_{1} \\
    V_{2} \\
  \end{bmatrix} T
    \begin{bmatrix}
    0 &
    V^T_{1} &
    V^T_{2} 
  \end{bmatrix}
\right) \begin{bmatrix}
    I \\
    & I \\
    & & \hat Q_1 \cdots \hat Q_r \\
  \end{bmatrix} \nonumber \\
  &=&B\left( \begin{bmatrix}
    I \\
    & I \\
    & & \hat Q_1 \cdots \hat Q_r \\
  \end{bmatrix} - 
  \begin{bmatrix}
    0 \\
    V_{1} \\
    V_{2} \\
  \end{bmatrix} T
    \begin{bmatrix}
    0 &
    V^T_{1} &
    \hat L^T
  \end{bmatrix}
\right) \nonumber \\
&=&  B \begin{bmatrix}
    I \\
    & I \\
    & & \hat Q_1 \cdots \hat Q_r \\
  \end{bmatrix} \left( I - \begin{bmatrix}
    0 \\
    V_{1} \\
    \hat L \\
  \end{bmatrix} T
    \begin{bmatrix}
    0 &
    V^T_{1} &
    \hat L^T
  \end{bmatrix}
\right). \label{eq:commuteQ}
\end{eqnarray}
Hence, the orthogonal transformations nearly commute with the reflectors, but $V_2$ turns into $\hat L$.
The shape of the correspondingly modified matrix $V$ is displayed in Figure~\ref{fig:right-incorporate}~(b).

Additionally exploiting the shape of $\hat L$, see~\eqref{eq:lqdecomposition}, we update columns $s+1:n$ of $B$ according to~\eqref{eq:commuteQ} as follows:
\begin{enumerate}
 \item $B_{:,j+1:n} \gets B_{:,j+1:n} \hat Q_{1} \cdots \hat Q_{r}$, 
 \item $W \gets B_{:,s+1:j} V_1 + B_{:,n-k+1:n} \hat L_{1}$, 
 \item $B_{:,s+1:j} \gets B_{:,s+1:j} - W T V_{1}^{T}$,
 \item $B_{:,n-k+1:n} \gets B_{:,n-k+1:n} - W T \hat L_{1}^{T}$.
\end{enumerate}
\begin{figure}
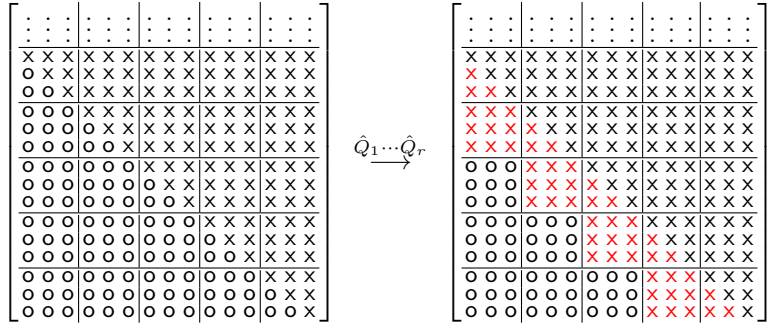

 \[
\left[\arraycolsep=1.4pt\renewcommand{\arraystretch}{0.6} 
\small \begin{array}{ccc|ccc|ccc|ccc|ccc}
\vdots & \vdots & \vdots & \vdots & \vdots & \vdots & \vdots & \vdots & \vdots & \vdots & \vdots & \vdots & \vdots & \vdots & \vdots \\ \hline  \\[-6pt]
\blackx & \blackx & \blackx & \blackx & \blackx & \blackx & \blackx & \blackx & \blackx & \blackx & \blackx & \blackx& \blackx & \blackx & \blackx \\
\blacko & \blackx & \blackx & \blackx & \blackx & \blackx & \blackx & \blackx & \blackx & \blackx & \blackx & \blackx& \blackx & \blackx & \blackx \\
\blacko & \blacko & \blackx & \blackx & \blackx & \blackx & \blackx & \blackx & \blackx & \blackx & \blackx & \blackx& \blackx & \blackx & \blackx \\ \hline  \\[-6pt]
\blacko & \blacko & \blacko & \blackx & \blackx & \blackx & \blackx & \blackx & \blackx & \blackx & \blackx & \blackx& \blackx & \blackx & \blackx \\
\blacko & \blacko & \blacko & \blacko & \blackx & \blackx & \blackx & \blackx & \blackx & \blackx & \blackx & \blackx& \blackx & \blackx & \blackx \\
\blacko & \blacko & \blacko & \blacko & \blacko & \blackx & \blackx & \blackx & \blackx & \blackx & \blackx & \blackx& \blackx & \blackx & \blackx \\ \hline  \\[-6pt]
\blacko & \blacko & \blacko & \blacko & \blacko & \blacko & \blackx & \blackx & \blackx & \blackx & \blackx & \blackx& \blackx & \blackx & \blackx \\
\blacko & \blacko & \blacko & \blacko & \blacko & \blacko & \blacko & \blackx & \blackx & \blackx & \blackx & \blackx& \blackx & \blackx & \blackx \\
\blacko & \blacko & \blacko & \blacko & \blacko & \blacko & \blacko & \blacko & \blackx & \blackx & \blackx & \blackx& \blackx & \blackx & \blackx \\ \hline  \\[-6pt]
\blacko & \blacko & \blacko & \blacko & \blacko & \blacko & \blacko & \blacko & \blacko & \blackx & \blackx & \blackx& \blackx & \blackx & \blackx \\
\blacko & \blacko & \blacko & \blacko & \blacko & \blacko & \blacko & \blacko & \blacko & \blacko & \blackx & \blackx& \blackx & \blackx & \blackx \\
\blacko & \blacko & \blacko & \blacko & \blacko & \blacko & \blacko & \blacko & \blacko & \blacko & \blacko & \blackx& \blackx & \blackx & \blackx \\ \hline  \\[-6pt]
\blacko & \blacko & \blacko & \blacko & \blacko & \blacko & \blacko & \blacko & \blacko & \blacko & \blacko & \blacko& \blackx & \blackx & \blackx \\
\blacko & \blacko & \blacko & \blacko & \blacko & \blacko & \blacko & \blacko & \blacko & \blacko & \blacko & \blacko& \blacko & \blackx & \blackx \\
\blacko & \blacko & \blacko & \blacko & \blacko & \blacko & \blacko & \blacko & \blacko & \blacko & \blacko & \blacko& \blacko & \blacko & \blackx \\
\end{array}
\right] \ \stackrel{\hat Q_{1} \cdots \hat Q_{r}}{\longrightarrow}\ \left[\arraycolsep=1.4pt\renewcommand{\arraystretch}{0.6} 
\small \begin{array}{ccc|ccc|ccc|ccc|ccc}
\vdots & \vdots & \vdots & \vdots & \vdots & \vdots & \vdots & \vdots & \vdots & \vdots & \vdots & \vdots & \vdots & \vdots & \vdots \\ \hline  \\[-6pt]
\blackx & \blackx & \blackx & \blackx & \blackx & \blackx & \blackx & \blackx & \blackx & \blackx & \blackx & \blackx& \blackx & \blackx & \blackx \\
\redx & \blackx & \blackx & \blackx & \blackx & \blackx & \blackx & \blackx & \blackx & \blackx & \blackx & \blackx& \blackx & \blackx & \blackx \\
\redx & \redx & \blackx & \blackx & \blackx & \blackx & \blackx & \blackx & \blackx & \blackx & \blackx & \blackx& \blackx & \blackx & \blackx \\ \hline  \\[-6pt]
\redx & \redx & \redx & \blackx & \blackx & \blackx & \blackx & \blackx & \blackx & \blackx & \blackx & \blackx& \blackx & \blackx & \blackx \\
\redx & \redx & \redx & \redx & \blackx & \blackx & \blackx & \blackx & \blackx & \blackx & \blackx & \blackx& \blackx & \blackx & \blackx \\
\redx & \redx & \redx & \redx & \redx & \blackx & \blackx & \blackx & \blackx & \blackx & \blackx & \blackx& \blackx & \blackx & \blackx \\ \hline  \\[-6pt]
\blacko & \blacko & \blacko & \redx & \redx & \redx & \blackx & \blackx & \blackx & \blackx & \blackx & \blackx& \blackx & \blackx & \blackx \\
\blacko & \blacko & \blacko & \redx & \redx & \redx & \redx & \blackx & \blackx & \blackx & \blackx & \blackx& \blackx & \blackx & \blackx \\
\blacko & \blacko & \blacko & \redx & \redx & \redx & \redx & \redx & \blackx & \blackx & \blackx & \blackx& \blackx & \blackx & \blackx \\ \hline  \\[-6pt]
\blacko & \blacko & \blacko & \blacko & \blacko & \blacko & \redx & \redx & \redx & \blackx & \blackx & \blackx& \blackx & \blackx & \blackx \\
\blacko & \blacko & \blacko & \blacko & \blacko & \blacko & \redx & \redx & \redx & \redx & \blackx & \blackx& \blackx & \blackx & \blackx \\
\blacko & \blacko & \blacko & \blacko & \blacko & \blacko & \redx & \redx & \redx & \redx & \redx & \blackx& \blackx & \blackx & \blackx \\ \hline  \\[-6pt]
\blacko & \blacko & \blacko & \blacko & \blacko & \blacko & \blacko & \blacko & \blacko & \redx & \redx & \redx & \blackx & \blackx & \blackx \\
\blacko & \blacko & \blacko & \blacko & \blacko & \blacko & \blacko & \blacko & \blacko & \redx & \redx & \redx & \redx & \blackx & \blackx \\
\blacko & \blacko & \blacko & \blacko & \blacko & \blacko & \blacko & \blacko & \blacko & \redx & \redx & \redx & \redx & \redx & \blackx \\
\end{array}
\right]
\]
\caption{ \label{fig:fillinB}  Shape of $B_{:,j+1:n} \hat Q_{1} \cdots \hat Q_{r}$.}
\end{figure}
In Step~1, the application of $\hat Q_{1} \cdots \hat Q_{r}$ involves multiplying $B$ with $2k\times 2k$ orthogonal matrices (in terms of their WY representations) from the right.
This will update columns $j+1:n$.
Note that this will transform the structure of $B$ as illustrated in Figure~\ref{fig:fillinB}.
Step~3 introduces fill-in in columns $s+1:j$ while Step~4 does not introduce additional fill-in.
In summary, the transformed matrix $B$ takes the form sketched in Figure~\ref{fig:right-incorporate}~(c).
\end{paragraph}

\begin{paragraph}{d) Apply orthogonal transformations to $Z$} 
Replacing $B$ by $Z$ in~\eqref{eq:commuteQ}, the update of columns $s+1:n$ of $Z$ takes the following form:
\begin{enumerate}
 \item $Z_{:,j+1:n} \gets Z_{:,j+1:n} \hat Q_{1} \cdots \hat Q_{r}$,
 \item $W \gets Z_{:,s+1:j} V_1  + Z_{:,n-k+1:n} \hat L_{1}$, 
 \item $Z_{:,s+1:j} \gets Z_{:,s+1:j} - WT V_{1}^{T}$,
 \item $Z_{:,n-k+1:n} \gets Z_{:,n-k+1:n} - W T \hat L_{1}^{T}$.
\end{enumerate}
\end{paragraph}

\begin{paragraph}{e) Apply orthogonal transformations to $A$} 
The update of $A$ is slightly different due to the presence of the intermediate matrix $Y = A V T$ and the panel which is already reduced.
However, the basic idea remains the same.
After post-multiplying with $\hat Q_{1} \cdots \hat Q_{r}$ we get
\begin{eqnarray*}
   A &\gets& \left(
    A
    -
    Y
    \begin{bmatrix}
      0 & V_{1}^{T} & V_{2}^{T}
    \end{bmatrix}
  \right)
  \begin{bmatrix}
    I \\
    & I \\
    & & \hat Q_1 \cdots \hat Q_r \\
  \end{bmatrix} \\ &=& 
  A
  \begin{bmatrix}
    I \\
    & I \\
    & & \hat Q_1 \cdots \hat Q_r \\
  \end{bmatrix}
  -
  Y
  \begin{bmatrix}
    0 & V_{1}^{T} & \hat L^{T}
  \end{bmatrix}.
\end{eqnarray*}
The first $j-1$ columns of $A$ have already been updated (Fact~\ref{fact4}) but column $j$ still needs to be updated.
We arrive at the following procedure for updating $A$:
\begin{enumerate}
 \item $A_{:,j+1:n} \gets A_{:,j+1:n} \hat Q_{1} \cdots \hat Q_{r}$,
 \item $A_{:,j} \gets A_{:,j} - Y (V_{1})_{k,:}^{T}$,
 \item $A_{:,n-k+1:n} \gets A_{:,n-k+1:n} - Y \hat L_{1}^{T}$.
\end{enumerate}
\end{paragraph}

\begin{paragraph}{e) Partially restore the triangular shape of $B$} 
The absorption of the right reflectors is completed by reducing the last $n-j$ columns of $B$ back to triangular form via a sequence of RQ decompositions from bottom to top. This starts with an RQ decomposition of $B_{n-k+1:n,n-2k+1:n}$.
After updating columns $n-2k+1:n$ of $B$ with the corresponding orthogonal transformation $\tilde Q_1$, we proceed with an RQ decomposition of $B_{n-2k+1:n-k,n-3k+1:n-k}$, and so on, until all sub-diagonal blocks of $B_{:,j+1:n}$ (see Figure~\ref{fig:fillinB}) have been processed.
The resulting orthogonal transformation matrices $\tilde Q_1,\ldots,\tilde Q_r$ are multiplied into $A$ and $Z$ as well:
\[
 \begin{array}{rcl} 
 A_{ :, j + 1 : n } &\gets & A_{ :, j + 1 : n } \tilde Q_{1}^{T} \tilde Q_{2}^{T} \cdots \tilde Q_{r}^{T}, \\
 Z_{ :, j + 1 : n } &\gets& Z_{ :, j + 1 : n } \tilde Q_{1}^{T} \tilde Q_{2}^{T} \cdots \tilde Q_{r}^{T}.
\end{array}
\] 
The shape of $B$ after this procedure is displayed in Figure~\ref{fig:right-incorporate}~(d).
\end{paragraph}

\subsubsection{Absorption of left reflectors} \label{sec:absorbleft}

We now turn our attention to the absorption of the left reflectors $I - U S U^{T}$ into $A$, $B$, and $Q$.
When doing so we are free to apply additional transformations from left \change{(to introduce zeros in $U$)} or right \change{(to introduce zeros in $W = SU^{T}$)}.
Because of the reduced forms of $A$ and $B$, it is cheaper to apply transformations from the left.
The ideas and techniques are quite similar to what has been described in Section~\ref{sec:absorbright} for absorbing right reflectors, and we therefore keep the following description brief.

\begin{paragraph}{a) Initial situation}
We partition $U$ as 
$
  U =\small 
  \begin{bmatrix}
    0 \\
    U_{1} \\
    U_{2} \\
  \end{bmatrix}
$, where $U_{1}$ is a $k \times k$ lower triangular matrix starting at row $s + 1$ (Fact~\ref{fact2}).
\end{paragraph}

\begin{paragraph}{b) Reduce $U$} We reduce the matrix $U_2$ to upper triangular form by a sequence of $r \change{=} (n-j-k)/k$ QR decompositions as illustrated in the following diagram:
\[
\left[\arraycolsep=1.4pt\renewcommand{\arraystretch}{0.6} 
\small \begin{array}{ccc}
\blackx & \blackx & \blackx \\
\blackx & \blackx & \blackx \\
\blackx & \blackx & \blackx \\ \hline  \\[-6pt]
\blackx & \blackx & \blackx \\
\blackx & \blackx & \blackx \\
\blackx & \blackx & \blackx \\ \hline  \\[-6pt]
\blackx & \blackx & \blackx \\
\blackx & \blackx & \blackx \\
\blackx & \blackx & \blackx \\ \hline  \\[-6pt]
\blackx & \blackx & \blackx \\
\blackx & \blackx & \blackx \\
\blackx & \blackx & \blackx \\ \hline  \\[-6pt]
\blackx & \blackx & \blackx \\
\blackx & \blackx & \blackx \\
\blackx & \blackx & \blackx 
\end{array}
\right] \stackrel{\tilde Q_1}{\longrightarrow}
\left[\arraycolsep=1.4pt\renewcommand{\arraystretch}{0.6} 
\small \begin{array}{ccc}
\blackx & \blackx & \blackx \\
\blackx & \blackx & \blackx \\
\blackx & \blackx & \blackx \\ \hline  \\[-6pt]
\blackx & \blackx & \blackx \\
\blackx & \blackx & \blackx \\
\blackx & \blackx & \blackx \\ \hline  \\[-6pt]
\blackx & \blackx & \blackx \\
\blackx & \blackx & \blackx \\
\blackx & \blackx & \blackx \\ \hline  \\[-6pt]
\redx & \redx & \redx \\ 
\blueo & \redx & \redx  \\
\blueo & \blueo & \redx \\ \hline  \\[-6pt]
\blueo & \blueo & \blueo \\
\blueo & \blueo & \blueo \\
\blueo & \blueo & \blueo \\ 
\end{array}
\right] \stackrel{\tilde Q_2}{\longrightarrow}
\left[\arraycolsep=1.4pt\renewcommand{\arraystretch}{0.6} 
\small \begin{array}{ccc}
\blackx & \blackx & \blackx \\
\blackx & \blackx & \blackx \\
\blackx & \blackx & \blackx \\ \hline  \\[-6pt]
\blackx & \blackx & \blackx \\
\blackx & \blackx & \blackx \\
\blackx & \blackx & \blackx \\ \hline  \\[-6pt]
\redx & \redx & \redx \\ 
\blueo & \redx & \redx  \\
\blueo & \blueo & \redx \\ \hline  \\[-6pt]
\blueo & \blueo & \blueo \\
\blacko & \blueo & \blueo \\
\blacko & \blacko & \blueo \\ \hline  \\[-6pt]
\blacko & \blacko & \blacko \\
\blacko & \blacko & \blacko \\
\blacko & \blacko & \blacko \\ 
\end{array}
\right] \quad \cdots \quad \stackrel{\tilde Q_r}{\longrightarrow}
\left[\arraycolsep=1.4pt\renewcommand{\arraystretch}{0.6} 
\small \begin{array}{ccc}
\redx & \redx & \redx \\ 
\blueo & \redx & \redx  \\
\blueo & \blueo & \redx \\ \hline  \\[-6pt]
\blueo & \blueo & \blueo \\
\blacko & \blueo & \blueo \\
\blacko & \blacko & \blueo \\ \hline  \\[-6pt]
\blacko & \blacko & \blacko \\
\blacko & \blacko & \blacko \\
\blacko & \blacko & \blacko \\ \hline  \\[-6pt] 
\blacko & \blacko & \blacko \\
\blacko & \blacko & \blacko \\
\blacko & \blacko & \blacko \\ \hline  \\[-6pt] 
\blacko & \blacko & \blacko \\
\blacko & \blacko & \blacko \\
\blacko & \blacko & \blacko \\ 
\end{array}
\right].
\]
This corresponds to a decomposition of the form
\begin{equation} \label{eq:qrdecomposition}
U_{2} = \tilde Q_{1} \cdots \tilde Q_{r} \tilde R \quad \text{with} \quad \tilde R = \begin{bmatrix} \tilde R_1 \\ 0 \end{bmatrix},
\end{equation}
where $\tilde R_1$ is a $k\times k$ upper triangular matrix.
\end{paragraph}

\begin{paragraph}{c) Apply orthogonal transformations to $B$}
We first update columns $s+1:j$ of $B$, corresponding to the ``spike'' shown in Figure~\ref{fig:right-incorporate}~(d):
\begin{enumerate}
 \item $B_{s+1:j,s+1:j} \gets B_{s+1:j,s+1:j} - U_1 S^T \begin{bmatrix} U_1^T & U_2^T \end{bmatrix} B_{s+1:n,s+1:j}$,
 \item $B_{j+1:n,s+1:j} \gets 0$.
\end{enumerate}
Here, we use that columns $s+1:j$ are guaranteed to be in triangular form after the application of the right and left reflectors (Fact~\ref{fact5}).

For the remaining columns, we multiply with $\tilde Q_r^T \cdots \tilde Q_1^T$ from the left and get
 \begin{eqnarray}
   B &\gets& \begin{bmatrix}
    I \\
    & I \\
    & & \tilde Q_r^T \cdots \tilde Q_1^T \\
  \end{bmatrix}
  \left(
    I -
    \begin{bmatrix}
      0 \\
      U_{1} \\
      U_{2} \\
    \end{bmatrix}
    S^{T}
    \begin{bmatrix}
      0 & U_{1}^{T} & U_{2}^{T}
    \end{bmatrix}
  \right) B  \nonumber \\
  &=&   \left(
    \begin{bmatrix}
    I \\
    & I \\
    & & \tilde Q_r^T \cdots \tilde Q_1^T \\
  \end{bmatrix} -
    \begin{bmatrix}
      0 \\
      U_{1} \\
      \tilde R \\
    \end{bmatrix}
    S^{T}
    \begin{bmatrix}
      0 & U_{1}^{T} & U_{2}^{T}
    \end{bmatrix}
  \right) B \nonumber \\
  &=&   \left(
    I -
    \begin{bmatrix}
      0 \\
      U_{1} \\
      \tilde R \\
    \end{bmatrix}
    S^{T}
    \begin{bmatrix}
      0 & U_{1}^{T} & \tilde R^{T}
    \end{bmatrix}
  \right)\begin{bmatrix}
    I \\
    & I \\
    & & \tilde Q_r^T \cdots \tilde Q_1^T \\
  \end{bmatrix} B. \label{eq:commuteQ2}
\end{eqnarray}
Additionally exploiting the shape of $\tilde R$, see~\eqref{eq:qrdecomposition}, we update columns $j+1:n$ of $B$ according to~\eqref{eq:commuteQ2} as follows:
\begin{enumerate} \setcounter{enumi}{2}
 \item $B_{j+1:n,s+1:n} \gets \tilde Q_r^T \cdots \tilde Q_1^T B_{j+1:n,s+1:n}$,
 \item $W \gets B_{s+1:j+k,j+1:n}^T     \begin{bmatrix}
      U_{1} \\
      \tilde R_1 \\
    \end{bmatrix}$,
 \item $B_{s+1:j+k,j+1:n} \gets B_{s+1:j+k,j+1:n} - \begin{bmatrix}
      U_{1} \\
      \tilde R_1 \\
    \end{bmatrix}S^T W^T$.
\end{enumerate}
The triangular shape of $B_{j+1:n,j+1:n}$ is exploited in Step~3 and gets transformed into the shape shown in Figure~\ref{fig:fillinB}.  
\end{paragraph}

\begin{paragraph}{d) Apply orthogonal transformations to $Q$}
  Replace $B$ with $Q$ in (\ref{eq:commuteQ2}) and get
\begin{enumerate}
 \item $Q_{:,j+1:n} \gets Q_{:,j+1:n} \tilde Q_1 \cdots \tilde Q_r$,
 \item $W \gets Q_{:,s+1:j+k}     \begin{bmatrix}
      U_{1} \\
      \tilde R_1 \\
    \end{bmatrix}$,
 \item $Q_{:,s+1:j+k} \gets Q_{:,s+1:j+k} - W S \begin{bmatrix}
      U_{1}^T & \tilde R_1^T 
    \end{bmatrix}$.
\end{enumerate} 
\end{paragraph}

\begin{paragraph}{e) Apply orthogonal transformations to $A$} Exploiting that the first $j-1$ columns of $A$ are updated and zero below row $j$ (Fact~\ref{fact4}), the update of $A$ takes the form:
\begin{enumerate}
 \item $A_{j+1:n,j:n} \gets \tilde Q_r^T \cdots \tilde Q_1^T A_{j+1:n,j:n}$,
 \item $W \gets A_{s+1:j+k,j:n}^T     \begin{bmatrix}
      U_{1} \\
      \tilde R_1 \\
    \end{bmatrix}$,
 \item $A_{s+1:j+k,j:n} \gets A_{s+1:j+k,j:n} - \begin{bmatrix}
      U_{1} \\
      \tilde R_1 \\
    \end{bmatrix}S^T W^T$.
\end{enumerate} 
\end{paragraph}

\begin{paragraph}{f) Restore the triangular shape of $B$} 
At this point, the first $j$ columns of $B$ are in triangular form (see Part~c), while the last $n-j$ columns are not and take the form shown in Figure~\ref{fig:fillinB}, right.
We reduce columns $j+1:n$ of $B$ back to triangular form by a sequence of QR decompositions from top to bottom.
This starts with a QR decomposition of $B_{j+1:j+2k,j+1:j+k}$.
After updating rows $j+1:j+2k$ of $B$ with the corresponding orthogonal transformation $\hat Q_1$, we proceed with a QR decomposition of $B_{j+k+1:j+3k,j+k+1:j+2k}$, and so on, until all subdiagonal blocks of $B_{:,j+1:n}$ have been processed. The resulting orthogonal transformation matrices $\hat Q_1,\ldots,\hat Q_r$ are multiplied into $A$ and $Q$ as well:
\[
 \begin{array}{rcl} 
 A_{j + 1 : n,j:n } &\gets & \hat Q_{r}^{T} \cdots \hat Q_{2}^{T} \hat Q_{1}^{T} A_{j + 1 : n,j:n }, \\
 Q_{ :, j + 1 : n } &\gets& Q_{ :, j + 1 : n } \hat Q_{1} \hat Q_{2} \cdots \hat Q_{r}.
\end{array}
\] 
This completes the absorption of right and left reflectors.
\change{Note that the absorption procedure preserves the zero pattern of the first $j$ columns of $A$.}

\begin{remark} \label{remark:notdivide}
\change{The procedure needs to be only slightly adapted to correctly treat the case when $k$ does not divide $n-j$. The first QL (respectively, QR) factorization in the reduction of $V$ (respectively,~$U$), needs to involve $\tilde{k}$ instead of $2k$ rows, where $k < \tilde{k} < 2k$ is such that $k$ divides $n-j-\tilde{k}$. All the other factorizations still involve $2k$ rows, as described. Once the orthogonal transformations are applied to $B$, the first (respectively, last) subdiagonal block of $B$ will be of size different from the rest. Therefore, we will need to adapt the size of the submatrix for the last RQ (respectively, QR) factorization during the restoration of the triangular shape as well.}
\end{remark}
\end{paragraph}

\subsection{Summary of algorithm}

Summarizing the developments of this section, Algorithm~\ref{alg:main} gives the basic form of our newly proposed Householder-based method for reducing a matrix pencil $A-\lambda B$, with upper triangular $B$, to Hessenberg-triangular form.
The case of iterative refinement failures can be handled in different ways.
In Algorithm~\ref{alg:main} the last left reflector is explicitly undone, which is arguably the simplest approach.
In our implementation, we instead use an approach that avoids redundant computations at the expense of added complexity.
The differences in performance should be minimal.

\begin{algorithm2e}[htbp] 
  \caption{$[H, T, Q, Z] = \mathtt{HouseHT}(A, B)$}
  \label{alg:main}
  \tcp{Initialize}
  $Q \gets I$; $Z \gets I$\;
  Clear out $V$, $T$, $U$, $S$, $Y$\;
  $k \gets 0$\tcp*[l]{$k$ keeps track of the number of delayed reflectors}
  \tcp{For each column to reduce in $A$}
  \For{$j = 1:n-2$}
  {
    \tcp{Reduce column $j$ of $A$}
    Update column $j$ of $A$ from both sides w.r.t.~the $k$ delayed updates (see Section~\ref{sec:columnj}a)\;
    Reduce column $j$ of $A$ with a new reflector $I - \beta u u^{T}$ (see Section~\ref{sec:columnj}b)\;
    Augment $I - U S U^{T}$ with $I - \beta u u^{T}$ (see Section~\ref{sec:columnj}b)\;
    \tcp{Implicitly reduce column $j + 1$ of $B$}
    Attempt to solve the triangular system (see Section~\ref{sec:columnj}c) \change{for} vector $x$\;
    \eIf{the solve succeeded}
    {
      Reduce $x$ with a new reflector $I - \gamma v v^{T}$ (see Section~\ref{sec:columnj}d)\;
      Augment $I - V T V^{T}$ with $I - \gamma v v^{T}$ (see Section~\ref{sec:columnj}d)\;
      Augment $Y$ with $I - \gamma v v^{T}$ (see Section~\ref{sec:columnj}d)\;
      $k \gets k + 1$\;
    }
    {
      Undo the reflector $I - \beta u u^{T}$ by restoring the $j$th column of $A$, removing the last column of $U$, and removing the last row and column of $S$\;
    }
    \tcp{Absorb all reflectors}
    \If{$k = \nb$ {\bf or} the solve failed}
    {
      Absorb reflectors from the right (see Section~\ref{sec:absorbright})\;
      Absorb reflectors from the left (see Section~\ref{sec:absorbleft})\;
      Clear out $V$, $T$, $U$, $S$, $Y$\;
      $k \gets 0$\;
      % \If{the solve failed}{
      %   Redo the current loop iteration (do not increment $j$)\;
      % }
    }
  }
  \tcp{We are done}
  \Return $[A, B, Q, Z]$\;
\end{algorithm2e}

The algorithm has been designed to require $\Theta(n^3)$ flops.
Instead of a tedious derivation of the precise number of flops (which is further complicated by the occasional need for iterative refinement), we have measured this number experimentally; see Section~\ref{sec:experiments}.
\change{When applied to large random matrices (for which few iterative refinement iterations are necessary), {\tt HouseHT} requires roughly} 
$2.1 \pm 0.2$ 
times more flops than {\tt DGGHRD3}.
Note that on more difficult problems \change{(i.e., those requiring more iterations of iterative refinement and/or more prematurely triggered absorptions to maintain stability)} this factor will increase.

\subsection{Varia}
\label{sec:varia}

In this section, we discuss a couple of additions that we have made to the basic algorithm described above.
These modifications make the algorithm \change{more robust toward} some types of difficult inputs (Section~\ref{sec:preprocessing}) and also slightly reduces the number of flops required for absorption of reflectors (Section~\ref{sec:accelerated-reduction}).

\subsubsection{Preprocessing} \label{sec:preprocessing}

A number of applications, such as mechanical systems with constraints~\cite{Ilchmann2017} and discretized fluid flow problems~\cite{Heinkenschloss2008}, give rise to matrix pencils that feature a potentially large number of infinite eigenvalues. Often, many or even all of the infinite eigenvalues are induced by the sparsity of $B$. This can be exploited, before performing any reduction, to reduce the effective problem size for both the HT-reduction and the subsequent eigenvalue computation. As we will see in Section~\ref{sec:experiments}, such a preprocessing step is particularly beneficial to the newly proposed algorithm\change{:} the removal of infinite eigenvalues reduces the need for iterative refinement when solving linear systems with the matrix $B$.

We have implemented preprocessing for the case that $B$ has $\ell > 1$ zero columns. We choose an appropriate permutation matrix $Z_0$ such that the first $\ell$ columns of $B Z_0$ are zero. If $B$ is diagonal, we also set $Q_0 = Z_0$ to preserve the diagonal structure; otherwise we set $Q_0 = I$.
Letting $A_0 = Q_0^T A Z_0$, we compute a QR decomposition of its first $\ell$ columns: 
$A_0(:, 1:\ell) = Q_1 \big[ {A_{11} \atop 0} \big]$, where $Q_1$ is an $n \times n$ orthogonal matrix and $A_{11}$ is an $\ell \times \ell$ upper triangular matrix. Then
$$
	A_1 = (Q_0 Q_1)^T A Z_0 
		= \left[ \begin{array}{cc} A_{11} & A_{12} \\ 0 & A_{22} \end{array} \right],
	\quad
	B_1 = (Q_0 Q_1)^T B Z_0 
		= \left[ \begin{array}{cc} 0 & B_{12} \\ 0 & B_{22} \end{array} \right],
$$
where $A_{22}, B_{22} \in \R^{(n-\ell)\times (n-\ell)}$. Noting that the top left $\ell \times \ell$ part of $A_1-\lambda B_1$ is already in generalized Schur form, only the trailing part $A_{22} -\lambda B_{22}$ needs to be reduced to Hessenberg-triangular form.

\subsubsection{Accelerated reduction of $V_{2}$ and $U_{2}$}
\label{sec:accelerated-reduction}

As we will see in the numerical experiments in Section~\ref{sec:experiments} below, Algorithm~\ref{alg:main} spends a significant fraction of the total execution time on the absorption of reflectors.
Inspired by techniques developed in~\cite[Sec.~2.2]{Kagstrom2008} for reducing a matrix pencil to \emph{block} Hessenberg-triangular form, we now describe a modification of the algorithms described in Sections~\ref{sec:absorbright} and~\ref{sec:absorbleft} that attains better performance by reducing the number of flops. % by a small constant factor.
We first describe the case when absorption takes place after accumulating $\nb$ reflectors and then briefly discuss the case when absorption takes place after an iterative refinement failure.

\begin{paragraph}{Reduction of $V_{2}$}
We first consider the reduction of $V_{2}$ from Section~\ref{sec:absorbright}~b) and partition $B$, $V_{2}$ into blocks of size $\nb \times \nb$ as indicated in Figure~\ref{fig:accelrq1}~(a).
Recall that the algorithm for reducing $V_{2}$ proceeds by computing a sequence of QL decompositions of \emph{two} adjacent blocks.
Our proposed modification computes QL decompositions of $\ell \geq 3$ adjacent blocks at a time.
Figure~\ref{fig:accelrq1}~(b)--(d) illustrates this process for $\ell = 3$, showing how the reduction of 
$V_{2}$ affects $B$ when updating it with the corresponding  transformations from the right.
Compared to Figure~\ref{fig:fillinB}, the fill-in increases from overlapping $2 \nb \times 2 \nb$ blocks to overlapping $\ell \nb \times \ell \nb$ blocks on the diagonal.
For a matrix $V_{2}$ of size $n\times \nb$, the modified algorithm involves around $(n-\nb)/(\ell-1)\nb$ transformations, each corresponding to a WY representation of size $\ell\nb \times \nb$.
This compares favorably with the original algorithm which involves around $(n-\nb)/\nb$ WY representations of size $2\nb \times \nb$.
For $\ell = 3$ this implies that the overall cost of applying WY representations is reduced by between $10\%$ and $25\%$, depending on how much of their triangular structure is exploited; see also~\cite{Kagstrom2008}. These reductions quickly flatten out when increasing $\ell$ further. (Our implementation uses $\ell = 4$, which we found to be nearly optimal for the matrix sizes and computing environments considered in Section~\ref{sec:experiments}.) To keep the rest of the exposition simple, we focus on the case $\ell = 3$; the generalization to larger $\ell$ is straightforward. 

\begin{figure}[htb] 
  
  \subfloat[Initial configuration.]
  {
    \begin{tikzpicture}
      [scale=0.4,
      y=-1cm]

      \begin{scope}
        [xshift=-2cm,yshift=-14cm]

        \node [above] at (3.5,0) {$B$};
        \draw [fill=green!20] (0,0) -- (7,0) -- (7,7) -- cycle;
        \draw (1,1) -- (7,1); \draw (1,1) -- (1,0);
        \draw (2,2) -- (7,2); \draw (2,2) -- (2,0);
        \draw (3,3) -- (7,3); \draw (3,3) -- (3,0);
        \draw (4,4) -- (7,4); \draw (4,4) -- (4,0);
        \draw (5,5) -- (7,5); \draw (5,5) -- (5,0);
        \draw (6,6) -- (7,6); \draw (6,6) -- (6,0);
      \end{scope}
      
      \begin{scope}[xshift=6cm,yshift=-14cm]
        \node [above] at (0.5,0) {$V_{2}$};

        \draw [fill=green!20] (0,0) -- (0,7) -- (1,7) -- (1,0) -- cycle;

        \draw [fill=red!20] (0,0) rectangle +(1,3);

        \draw (0,1) -- (1,1);
        \draw (0,2) -- (1,2);
        \draw (0,3) -- (1,3);
        \draw (0,4) -- (1,4);
        \draw (0,5) -- (1,5);
        \draw (0,6) -- (1,6);
      \end{scope}
    \end{tikzpicture}}\ 
  \subfloat[1st reduction step.]
  {
    \begin{tikzpicture}
      [scale=0.4,
      y=-1cm]

      \begin{scope}
        [xshift=-2cm,yshift=-14cm]

        \node [above] at (3.5,0) {$B$};
        \draw [fill=green!20] (0,0) -- (0,3) -- (3,3) -- (7,7) -- (7,0) -- cycle;
        \draw (0,1) -- (7,1); \draw (1,3) -- (1,0);
        \draw (0,2) -- (7,2); \draw (2,3) -- (2,0);
        \draw (3,3) -- (7,3); \draw (3,3) -- (3,0);
        \draw (4,4) -- (7,4); \draw (4,4) -- (4,0);
        \draw (5,5) -- (7,5); \draw (5,5) -- (5,0);
        \draw (6,6) -- (7,6); \draw (6,6) -- (6,0);

        \fill [pattern=north east lines, pattern color=black!50] (0,0) rectangle +(3,3);
        
        \draw [very thick] (0,0) rectangle +(3,3);
      \end{scope}
      
      \begin{scope}[xshift=6cm,yshift=-14cm]
        \node [above] at (0.5,0) {$V_{2}$};
        \draw [fill=white!20] (0,0) -- (0,7) -- (1,7) -- (1,0) -- cycle;
        \draw [fill=green!20] (0,2) -- (0,7) -- (1,7) -- (1,3) -- cycle;
        \draw [fill=red!20] (0,2) -- ++(1,1) -- ++(0,2) -- ++(-1,0) -- cycle;
        \draw (0,1) -- (1,1);
        \draw (0,2) -- (1,2);
        \draw (0,3) -- (1,3);
        \draw (0,4) -- (1,4);
        \draw (0,5) -- (1,5);
        \draw (0,6) -- (1,6);

        \fill [pattern=north east lines, pattern color=black!50] (0,0) rectangle +(1,3);
      \end{scope}
    \end{tikzpicture}}\ 
  \subfloat[2nd reduction step.]
  {
    \begin{tikzpicture}
      [scale=0.4,
      y=-1cm]

      \begin{scope}
        [xshift=-2cm,yshift=-14cm]

        \node [above] at (3.5,0) {$B$};
        \draw [fill=green!20] (0,0) -- (0,3) -- (2,3) -- (2,5) -- (5,5) -- (7,7) -- (7,0) -- cycle;
        \draw (0,1) -- (7,1); \draw (1,3) -- (1,0);
        \draw (0,2) -- (7,2); \draw (2,3) -- (2,0);
        \draw (2,3) -- (7,3); \draw (3,5) -- (3,0);
        \draw (2,4) -- (7,4); \draw (4,5) -- (4,0);
        \draw (5,5) -- (7,5); \draw (5,5) -- (5,0);
        \draw (6,6) -- (7,6); \draw (6,6) -- (6,0);
        
        \fill [pattern=north east lines, pattern color=black!50] (2,0) rectangle +(3,5);

        \draw [very thick] (0,0) rectangle +(3,3);
        \draw [very thick] (2,2) rectangle +(3,3);
      \end{scope}
      
      \begin{scope}[xshift=6cm,yshift=-14cm]
        \node [above] at (0.5,0) {$V_{2}$};
        \draw [fill=white!20] (0,0) -- (0,7) -- (1,7) -- (1,0) -- cycle;
        \draw [fill=green!20] (0,4) -- (0,7) -- (1,7) -- (1,5) -- cycle;
        \draw [fill=red!20] (0,4) -- ++(1,1) -- ++(0,2) -- ++(-1,0) -- cycle;
        \draw (0,1) -- (1,1);
        \draw (0,2) -- (1,2);
        \draw (0,3) -- (1,3);
        \draw (0,4) -- (1,4);
        \draw (0,5) -- (1,5);
        \draw (0,6) -- (1,6);

        \fill [pattern=north east lines, pattern color=black!50] (0,2) rectangle +(1,3);
      \end{scope}
    \end{tikzpicture}}\ 
  \subfloat[3rd reduction step.]
  {
    \begin{tikzpicture}
      [scale=0.4,
      y=-1cm]

      \begin{scope}
        [xshift=-2cm,yshift=-14cm]

        \node [above] at (3.5,0) {$B$};
        \draw [fill=green!20] (0,0) -- (0,3) -- (2,3) -- (2,5) -- (4,5) -- (4,7) -- (7,7) -- (7,0) -- cycle;
        \draw (0,1) -- (7,1); \draw (1,3) -- (1,0);
        \draw (0,2) -- (7,2); \draw (2,3) -- (2,0);
        \draw (2,3) -- (7,3); \draw (3,5) -- (3,0);
        \draw (2,4) -- (7,4); \draw (4,5) -- (4,0);
        \draw (4,5) -- (7,5); \draw (5,7) -- (5,0);
        \draw (4,6) -- (7,6); \draw (6,7) -- (6,0);

        \fill [pattern=north east lines, pattern color=black!50] (4,0) rectangle +(3,7);

        \draw [very thick] (0,0) rectangle +(3,3);
        \draw [very thick] (2,2) rectangle +(3,3);
        \draw [very thick] (4,4) rectangle +(3,3);
      \end{scope}
      
      \begin{scope}[xshift=6cm,yshift=-14cm]
        \node [above] at (0.5,0) {$V_{2}$};
        \draw [fill=white!20] (0,0) -- (0,7) -- (1,7) -- (1,0) -- cycle;
        \draw [fill=green!20] (0,6) -- (0,7) -- (1,7) -- (1,7) -- cycle;
        \draw (0,1) -- (1,1);
        \draw (0,2) -- (1,2);
        \draw (0,3) -- (1,3);
        \draw (0,4) -- (1,4);
        \draw (0,5) -- (1,5);
        \draw (0,6) -- (1,6);

        \fill [pattern=north east lines, pattern color=black!50] (0,4) rectangle +(1,3);
      \end{scope}
    \end{tikzpicture}}

  \caption{
    Reduction of $V_{2}$ to lower triangular form by successive QL decompositions of $\ell = 3$ blocks and its effect on the shape of $B$.
    The diagonal patterns show what has been modified relative to the previous step.
    The thick lines aim to clarify the block structure.
    The red regions identify the sub-matrices of $V_{2}$ that will be reduced in the next step.
  }
  \label{fig:accelrq1}
\end{figure}
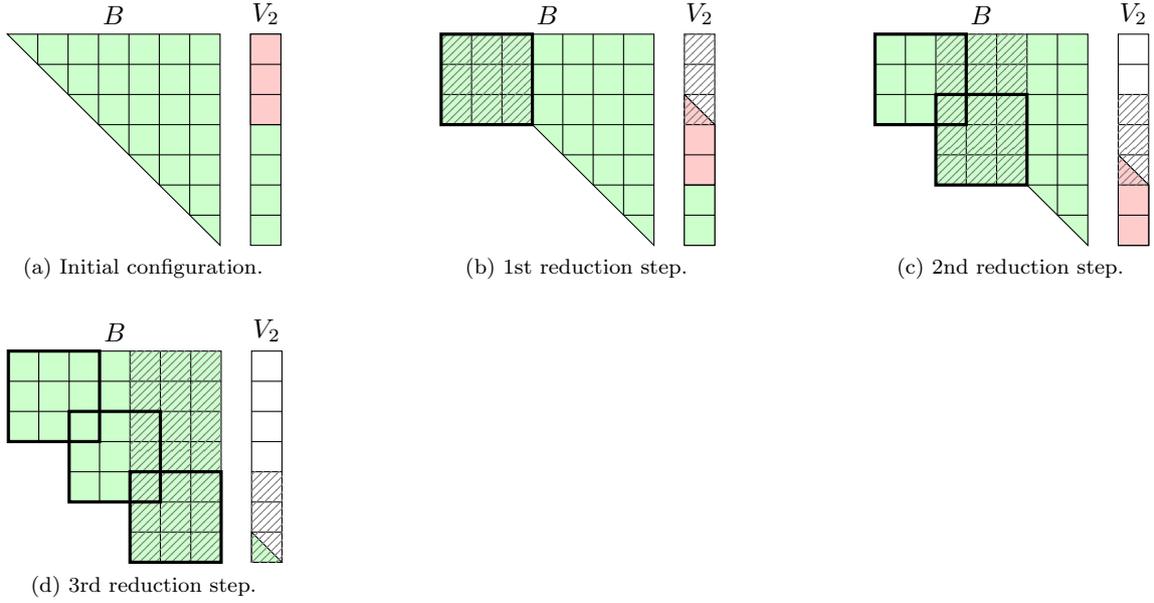
\end{paragraph}

\begin{paragraph}{Block triangular reduction of $B$ from the right}

After the reduction of $V_{2}$, we need to return $B$ to a form that facilitates the solution of linear systems with $B$ during the reduction of the next panel.
If we were to reduce the matrix $B$ in Figure~\ref{fig:accelrq1}~(d) fully back to triangular form then the advantages of the modification would be entirely consumed by this additional computational cost.
To avoid this, we reduce $B$ only to \emph{block} triangular form (with blocks of size $2 \nb \times 2 \nb$) using the following procedure.
Consider \change{an} RQ decomposition of an arbitrary $2\nb\times 3\nb$ matrix $C$:
\[
 C
= RQ = 
 \begin{bmatrix}
  0 & R_{12} & R_{13} \\
  0 & 0 & R_{23} 
 \end{bmatrix}
 \begin{bmatrix}
  Q_{11} & Q_{12} & Q_{13} \\
  Q_{21} & Q_{22} & Q_{23} \\
  Q_{31} & Q_{32} & Q_{33} \\
 \end{bmatrix}.
\]
Compute an LQ decomposition of the first block row of $Q$:
\[
E_1^T Q = 
 \begin{bmatrix}
  Q_{11} & Q_{12} & Q_{13} 
 \end{bmatrix} = 
  \begin{bmatrix}
  D_{11} & 0 & 0 \\
 \end{bmatrix} \tilde Q,
\]
where $E_1 = \begin{bmatrix} I_k & 0 & 0\end{bmatrix}^{T}$.
In other words, we have
\begin{displaymath}
  E_{1}^{T} Q \tilde Q^{T} =
  \begin{bmatrix}
    D_{11} & 0 & 0
  \end{bmatrix}
\end{displaymath}
with $D_{11}$ lower triangular.
Since the rows of this matrix are orthogonal and the matrix is triangular it must in fact be diagonal with diagonal entries $\pm 1$.
The first $\nb$ columns of $Q \tilde Q^{T}$ are orthogonal and each therefore has unit norm.
But since the top $\nb \times \nb$ block has $\pm 1$ on the diagonal,
% there is simply no room for any other non-zero entry on the same row and column of the matrix. In other words, 
\change{the off-diagonal entries in the first block row and column of $Q \tilde Q^T$ are zero. In particular,}
the first block column of $Q \tilde Q^{T}$ must be $E_{1} D_{11}$.
Thus, when applying $\tilde Q^T$ to $C$ from the right we obtain
\[
 C \tilde Q^T = RQ\tilde Q^T = 
 \begin{bmatrix}
  0 & R_{12} & R_{13} \\
  0 & 0 & R_{23} 
 \end{bmatrix}
 \begin{bmatrix}
  D_{11} & 0 & 0 \\
   0 & \hat Q_{22} & \hat Q_{23} \\
  0 & \hat Q_{32} & \hat Q_{33} \\
 \end{bmatrix} =  \begin{bmatrix}
  0 & \hat C_{12} & \hat C_{13} \\
  0 & \hat C_{22} & \hat C_{23} \\
 \end{bmatrix}.
\]
Note that multiplying with $\tilde Q^{T}$ from the \emph{right} reduces the first \emph{block column} of $C$.
Of course, the same effect could be attained with $Q$ but the key advantage of using $\tilde Q$ instead of $Q$ is that $\tilde Q$ consists of only $\nb$ reflectors with a WY representation of size $3\nb\times \nb$ compared with $Q$ which consists of $2\nb$ reflectors with a WY representation of size $3\nb \times 2\nb$.
This makes it significantly cheaper to apply $\tilde Q$ to other matrices.

Analogous constructions as those above can be made to efficiently reduce the \emph{last block row} of a $3\nb \times 2\nb$ matrix by multiplication from the \emph{left}.
Replace $C = RQ$ with $C = QR$ and replace the LQ decomposition of $E_{1}^{T} Q$ with a QL decomposition of $Q E_{3}$.
The matrix $\tilde Q^{T} Q$ will have special structure in its last block row and column (instead of the first block row and column). 

We apply the procedure described above%
\footnote{
  Our implementation actually computes RQ decompositions of full diagonal blocks (i.e., $3 \nb \times 3 \nb$ instead of $2 \nb \times 3 \nb$).
  The result is essentially the same but the performance is slightly worse. 
}
to $B$ in Figure~\ref{fig:accelrq2}~(a) starting at the bottom and obtain the shape shown in Figure~\ref{fig:accelrq2}~(b).
Continuing in this manner from bottom to top eventually yields a block triangular matrix with $2\nb\times 2\nb$ diagonal blocks, as shown in Figure~\ref{fig:accelrq2}~(a)--(d).
\begin{figure}[htb] 
  \centering
  \subfloat[Initial config.]
            {
  \begin{tikzpicture}
    [scale=0.4,
    y=-1cm]

    \begin{scope}
      [xshift=-2cm,yshift=-14cm]

      \draw [fill=green!20] (0,0) -- (0,3) -- (2,3) -- (2,5) -- (4,5) -- (4,7) -- (7,7) -- (7,0) -- cycle;

      \fill [red!20] (4,5) rectangle +(3,2);

      \draw (0,1) -- (7,1); \draw (1,3) -- (1,0);
      \draw (0,2) -- (7,2); \draw (2,3) -- (2,0);
      \draw (2,3) -- (7,3); \draw (3,5) -- (3,0);
      \draw (2,4) -- (7,4); \draw (4,5) -- (4,0);
      \draw (4,5) -- (7,5); \draw (5,7) -- (5,0);
      \draw (4,6) -- (7,6); \draw (6,7) -- (6,0);

      \draw [very thick] (0,0) rectangle +(3,3);
      \draw [very thick] (2,2) rectangle +(3,3);
      \draw [very thick] (4,4) rectangle +(3,3);
    \end{scope}
  \end{tikzpicture}}\ 
  \subfloat[1st reduction.]
            {
  \begin{tikzpicture}
    [scale=0.4,
    y=-1cm]

    \begin{scope}
      [xshift=-2cm,yshift=-14cm]

      \draw [fill=green!20] (0,0) -- (0,3) -- (2,3) -- (2,5) -- (5,5) -- (5,7) -- (7,7) -- (7,0) -- cycle;

      \fill [red!20] (2,3) rectangle +(3,2);

      \draw (0,1) -- (7,1); \draw (1,3) -- (1,0);
      \draw (0,2) -- (7,2); \draw (2,3) -- (2,0);
      \draw (2,3) -- (7,3); \draw (3,5) -- (3,0);
      \draw (2,4) -- (7,4); \draw (4,5) -- (4,0);
      \draw (4,5) -- (7,5); \draw (5,7) -- (5,0);
      \draw (5,6) -- (7,6); \draw (6,7) -- (6,0);

      \fill [pattern=north east lines, pattern color=black!50] (4,0) rectangle +(3,7);

      \draw [very thick] (0,0) rectangle +(3,3);
      \draw [very thick] (2,2) rectangle +(3,3);
      \draw [very thick] (5,5) rectangle +(2,2);
    \end{scope}
  \end{tikzpicture}}\ 
  \subfloat[2nd reduction.]
            {
  \begin{tikzpicture}
    [scale=0.4,
    y=-1cm]

    \begin{scope}
      [xshift=-2cm,yshift=-14cm]

      \draw [fill=green!20] (0,0) -- (0,3) -- (3,3) -- (3,5) -- (5,5) -- (5,7) -- (7,7) -- (7,0) -- cycle;

      \fill [red!20] (0,1) rectangle +(3,2);

      \draw (0,1) -- (7,1); \draw (1,3) -- (1,0);
      \draw (0,2) -- (7,2); \draw (2,3) -- (2,0);
      \draw (2,3) -- (7,3); \draw (3,5) -- (3,0);
      \draw (3,4) -- (7,4); \draw (4,5) -- (4,0);
      \draw (4,5) -- (7,5); \draw (5,7) -- (5,0);
      \draw (5,6) -- (7,6); \draw (6,7) -- (6,0);

      \fill [pattern=north east lines, pattern color=black!50] (2,0) rectangle +(3,5);

      \draw [very thick] (0,0) rectangle +(3,3);
      \draw [very thick] (3,3) rectangle +(2,2);
      \draw [very thick] (5,5) rectangle +(2,2);
    \end{scope}
  \end{tikzpicture}}\ 
  \subfloat[3rd reduction.]
            {
  \begin{tikzpicture}
    [scale=0.4,
    y=-1cm]

    \begin{scope}
      [xshift=-2cm,yshift=-14cm]

      \draw [fill=green!20] (0,0) -- (0,1) -- (1,1) -- (1,3) -- (3,3) -- (3,5) -- (5,5) -- (5,7) -- (7,7) -- (7,0) -- cycle;

      \draw (1,1) -- (7,1); \draw (1,3) -- (1,0);
      \draw (1,2) -- (7,2); \draw (2,3) -- (2,0);
      \draw (2,3) -- (7,3); \draw (3,5) -- (3,0);
      \draw (3,4) -- (7,4); \draw (4,5) -- (4,0);
      \draw (4,5) -- (7,5); \draw (5,7) -- (5,0);
      \draw (5,6) -- (7,6); \draw (6,7) -- (6,0);

      \fill [pattern=north east lines, pattern color=black!50] (0,0) rectangle +(3,3);

      \draw [very thick] (0,0) rectangle +(1,1);
      \draw [very thick] (1,1) rectangle +(2,2);
      \draw [very thick] (3,3) rectangle +(2,2);
      \draw [very thick] (5,5) rectangle +(2,2);
    \end{scope}
  \end{tikzpicture}}

\caption{Successive reduction of $B$ to block triangular form.
  The diagonal patterns show what has been modified from the previous configuration.
  The thick lines aim to clarify the block structure.
  The red regions identify the sub-matrices of $B$ that will be reduced in the next step.} \label{fig:accelrq2}
\end{figure}
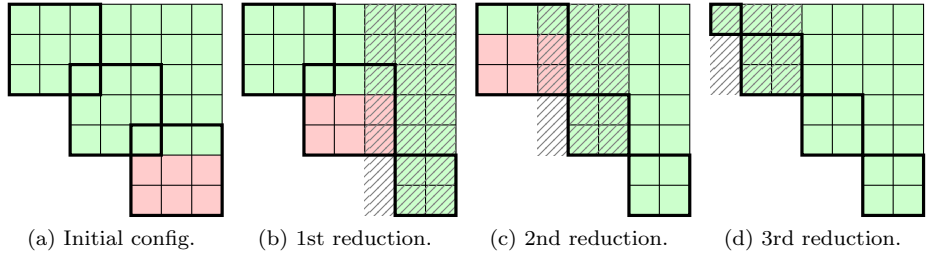

\end{paragraph}

\begin{paragraph}{Reduction of $U_{2}$}

  When absorbing reflectors from the left we reduce $U_{2}$ to upper triangular form as described in Section~\ref{sec:absorbleft}~b).
  The reduction of $U_{2}$ can be accelerated in much the same way as the reduction of $V_{2}$.
  However, since $B$ is \emph{block} triangular at this point, the tops of the sub-matrices of $U_{2}$ chosen for reduction must be aligned with the tops of the corresponding diagonal blocks of $B$.
  Figure~\ref{fig:accelqr1} gives a detailed example with proper alignment for $\ell = 3$.
  In particular, note that the first reduction uses a $2\nb \times \nb$ sub-matrix in order to align with the top of the first (i.e., bottom-most) diagonal block.
  Subsequent reductions use $3\nb \times \nb$ except the final reduction which is a special case. 
  
  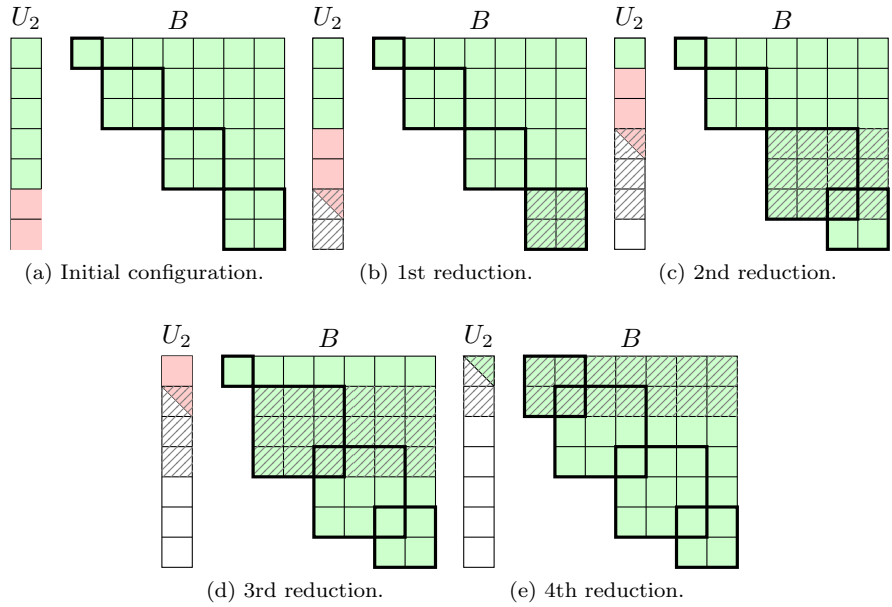
\begin{figure}[htb] 
    \centering
    \subfloat[Initial configuration.]
    {
      \begin{tikzpicture}
        [scale=0.4,
        y=-1cm]

        \begin{scope}
          [xshift=-2cm,yshift=-14cm]

          \node [above] at (3.5,0) {$B$};
          
          \draw [fill=green!20] (0,0) -- (0,1) -- (1,1) -- (1,3) -- (3,3) -- (3,5) -- (5,5) -- (5,7) -- (7,7) -- (7,0) -- cycle;

          \draw (1,1) -- (7,1); \draw (1,3) -- (1,0);
          \draw (1,2) -- (7,2); \draw (2,3) -- (2,0);
          \draw (2,3) -- (7,3); \draw (3,5) -- (3,0);
          \draw (3,4) -- (7,4); \draw (4,5) -- (4,0);
          \draw (4,5) -- (7,5); \draw (5,7) -- (5,0);
          \draw (5,6) -- (7,6); \draw (6,7) -- (6,0);

          \draw [very thick] (0,0) rectangle +(1,1);
          \draw [very thick] (1,1) rectangle +(2,2);
          \draw [very thick] (3,3) rectangle +(2,2);
          \draw [very thick] (5,5) rectangle +(2,2);
        \end{scope}

        \begin{scope}[xshift=-4cm,yshift=-14cm]

          \node [above] at (0.5,0) {$U_{2}$};
          \draw [fill=green!20] (0,0) rectangle +(1,7);

          \fill [fill=red!20] (0,5) rectangle +(1,2);
          
          \draw (0,1) -- (1,1);
          \draw (0,2) -- (1,2);
          \draw (0,3) -- (1,3);
          \draw (0,4) -- (1,4);
          \draw (0,5) -- (1,5);
          \draw (0,6) -- (1,6);
        \end{scope}
        
      \end{tikzpicture}
    }
    \subfloat[1st reduction.]
    {
      \begin{tikzpicture}
        [scale=0.4,
        y=-1cm]

        \begin{scope}
          [xshift=-2cm,yshift=-14cm]

          \node [above] at (3.5,0) {$B$};
          
          \draw [fill=green!20] (0,0) -- (0,1) -- (1,1) -- (1,3) -- (3,3) -- (3,5) -- (5,5) -- (5,7) -- (7,7) -- (7,0) -- cycle;

          \draw (1,1) -- (7,1); \draw (1,3) -- (1,0);
          \draw (1,2) -- (7,2); \draw (2,3) -- (2,0);
          \draw (2,3) -- (7,3); \draw (3,5) -- (3,0);
          \draw (3,4) -- (7,4); \draw (4,5) -- (4,0);
          \draw (4,5) -- (7,5); \draw (5,7) -- (5,0);
          \draw (5,6) -- (7,6); \draw (6,7) -- (6,0);

          \fill [pattern=north east lines, pattern color=black!50] (5,5) rectangle +(2,2);

          \draw [very thick] (0,0) rectangle +(1,1);
          \draw [very thick] (1,1) rectangle +(2,2);
          \draw [very thick] (3,3) rectangle +(2,2);
          \draw [very thick] (5,5) rectangle +(2,2);
        \end{scope}

        \begin{scope}[xshift=-4cm,yshift=-14cm]

          \node [above] at (0.5,0) {$U_{2}$};
          \draw (0,0) -- (0,7) -- (1,7) -- (1,0) -- cycle;
          \draw [fill=green!20] (0,5) -- ++(1,1) -- (1,0) -- (0,0) -- cycle;

          \fill [fill=red!20] (0,3) -- ++(1,0) -- ++(0,3) -- ++(-1,-1) -- cycle;
          
          \draw (0,1) -- (1,1);
          \draw (0,2) -- (1,2);
          \draw (0,3) -- (1,3);
          \draw (0,4) -- (1,4);
          \draw (0,5) -- (1,5);
          \draw (0,6) -- (1,6);

          \fill [pattern=north east lines, pattern color=black!50] (0,5) rectangle +(1,2);
        \end{scope}

      \end{tikzpicture}
    } 
    \subfloat[2nd reduction.]
    {
      \begin{tikzpicture}
        [scale=0.4,
        y=-1cm]

        \begin{scope}
          [xshift=-2cm,yshift=-14cm]

          \node [above] at (3.5,0) {$B$};
          
          \draw [fill=green!20] (0,0) -- (0,1) -- (1,1) -- (1,3) -- (3,3) -- (3,6) -- (5,6) -- (5,7) -- (7,7) -- (7,0) -- cycle;

          \draw (1,1) -- (7,1); \draw (1,3) -- (1,0);
          \draw (1,2) -- (7,2); \draw (2,3) -- (2,0);
          \draw (2,3) -- (7,3); \draw (3,5) -- (3,0);
          \draw (3,4) -- (7,4); \draw (4,6) -- (4,0);
          \draw (3,5) -- (7,5); \draw (5,7) -- (5,0);
          \draw (5,6) -- (7,6); \draw (6,7) -- (6,0);

          \fill [pattern=north east lines, pattern color=black!50] (3,3) rectangle +(4,3);

          \draw [very thick] (0,0) rectangle +(1,1);
          \draw [very thick] (1,1) rectangle +(2,2);
          \draw [very thick] (3,3) rectangle +(3,3);
          \draw [very thick] (5,5) rectangle +(2,2);
        \end{scope}
        
        \begin{scope}[xshift=-4cm,yshift=-14cm]

          \node [above] at (0.5,0) {$U_{2}$};
          \draw (0,0) -- (0,7) -- (1,7) -- (1,0) -- cycle;
          \draw [fill=green!20] (0,3) -- ++(1,1) -- (1,0) -- (0,0) -- cycle;

          \fill [fill=red!20] (0,1) -- ++(1,0) -- ++(0,3) -- ++(-1,-1) -- cycle;
          
          \draw (0,1) -- (1,1);
          \draw (0,2) -- (1,2);
          \draw (0,3) -- (1,3);
          \draw (0,4) -- (1,4);
          \draw (0,5) -- (1,5);
          \draw (0,6) -- (1,6);

          \fill [pattern=north east lines, pattern color=black!50] (0,3) rectangle +(1,3);
        \end{scope}

      \end{tikzpicture}
    } \\
    \subfloat[3rd reduction.]
    {
      \begin{tikzpicture}
        [scale=0.4,
        y=-1cm]

        \begin{scope}
          [xshift=-2cm,yshift=-14cm]

          \node [above] at (3.5,0) {$B$};
          
          \draw [fill=green!20] (0,0) -- (0,1) --(1,1) -- (1,4) -- (3,4) -- (3,6) -- (5,6) -- (5,7) -- (7,7) -- (7,0) -- cycle;

          \draw (1,1) -- (7,1); \draw (1,3) -- (1,0);
          \draw (1,2) -- (7,2); \draw (2,4) -- (2,0);
          \draw (1,3) -- (7,3); \draw (3,5) -- (3,0);
          \draw (3,4) -- (7,4); \draw (4,6) -- (4,0);
          \draw (3,5) -- (7,5); \draw (5,7) -- (5,0);
          \draw (5,6) -- (7,6); \draw (6,7) -- (6,0);

          \fill [pattern=north east lines, pattern color=black!50] (1,1) rectangle +(6,3);

          \draw [very thick] (0,0) rectangle +(1,1);
          \draw [very thick] (1,1) rectangle +(3,3);
          \draw [very thick] (3,3) rectangle +(3,3);
          \draw [very thick] (5,5) rectangle +(2,2);
        \end{scope}
              
        \begin{scope}[xshift=-4cm,yshift=-14cm]

          \node [above] at (0.5,0) {$U_{2}$};
          \draw (0,0) -- (0,7) -- (1,7) -- (1,0) -- cycle;
          \draw [fill=green!20] (0,1) -- ++(1,1) -- (1,0) -- (0,0) -- cycle;

          \fill [fill=red!20] (0,0) -- ++(1,0) -- ++(0,2) -- ++(-1,-1) -- cycle;
          
          \draw (0,1) -- (1,1);
          \draw (0,2) -- (1,2);
          \draw (0,3) -- (1,3);
          \draw (0,4) -- (1,4);
          \draw (0,5) -- (1,5);
          \draw (0,6) -- (1,6);

          \fill [pattern=north east lines, pattern color=black!50] (0,1) rectangle +(1,3);
        \end{scope}
        
      \end{tikzpicture}
    }
    \subfloat[4th reduction.]
    {
      \begin{tikzpicture}
        [scale=0.4,
        y=-1cm]

        \begin{scope}
          [xshift=-2cm,yshift=-14cm]

          \node [above] at (3.5,0) {$B$};
          
          \draw [fill=green!20] (0,0) -- (0,2) -- (1,2) -- (1,4) -- (3,4) -- (3,6) -- (5,6) -- (5,7) -- (7,7) -- (7,0) -- cycle;

          \draw (0,1) -- (7,1); \draw (1,3) -- (1,0);
          \draw (1,2) -- (7,2); \draw (2,4) -- (2,0);
          \draw (1,3) -- (7,3); \draw (3,5) -- (3,0);
          \draw (3,4) -- (7,4); \draw (4,6) -- (4,0);
          \draw (3,5) -- (7,5); \draw (5,7) -- (5,0);
          \draw (5,6) -- (7,6); \draw (6,7) -- (6,0);

          \fill [pattern=north east lines, pattern color=black!50] (0,0) rectangle +(7,2);

          \draw [very thick] (0,0) rectangle +(2,2);
          \draw [very thick] (1,1) rectangle +(3,3);
          \draw [very thick] (3,3) rectangle +(3,3);
          \draw [very thick] (5,5) rectangle +(2,2);
        \end{scope}
              
        \begin{scope}[xshift=-4cm,yshift=-14cm]

          \node [above] at (0.5,0) {$U_{2}$};
          \draw (0,0) -- (0,7) -- (1,7) -- (1,0) -- cycle;
          \draw [fill=green!20] (0,0) -- ++(1,1) -- (1,0) -- cycle;

          \draw (0,1) -- (1,1);
          \draw (0,2) -- (1,2);
          \draw (0,3) -- (1,3);
          \draw (0,4) -- (1,4);
          \draw (0,5) -- (1,5);
          \draw (0,6) -- (1,6);

          \fill [pattern=north east lines, pattern color=black!50] (0,0) rectangle +(1,2);
        \end{scope}
        
      \end{tikzpicture}
    }

    \caption{Reduction of $U_{2}$ to upper triangular form by successive QR decompositions and its effect on the shape of $B$.
      The diagonal patterns show what has been modified from the previous configuration.
      The thick lines aim to clarify the block structure.
      The red regions identify the sub-matrices of $U_{2}$ that will be reduced in the next step.} \label{fig:accelqr1}
  \end{figure}

\end{paragraph}

\begin{paragraph}{Block triangular reduction of $B$ from the left}

  The matrix $B$ must now be reduced back to block triangular form.
  The procedure is analogous to the one previously described but this time the transformations are applied from the left, and, once again, we have to be careful with the alignment of the blocks.
  Starting from the initial configuration illustrated in Figure~\ref{fig:accelqr2}~a) for $\ell=3$, the leading $2\nb \times \nb$ sub-matrix is fully reduced to upper triangular form.
  % These $\nb$ first columns form a part of the output and will not be modified any further.
  Subsequent steps of the reduction, illustrated in Figure~\ref{fig:accelqr2}~(b)--(d), use QR decompositions of $3\nb \times 2\nb$ sub-matrices to reduce the last $\nb$ rows of each block.
  
  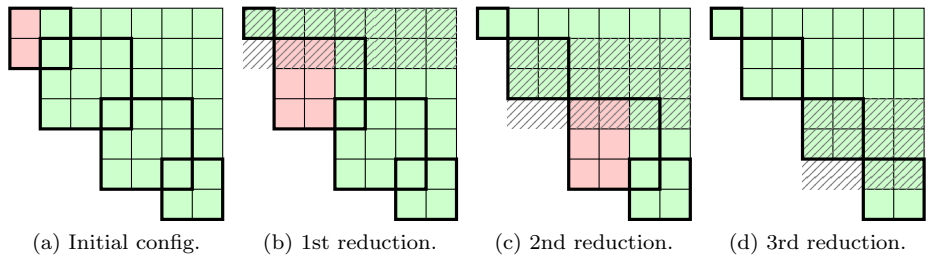
\begin{figure}[htb] 
    \centering
    \subfloat[Initial config.]
    {
      \begin{tikzpicture}
        [scale=0.4,
        y=-1cm]

        \begin{scope}
          [xshift=-2cm,yshift=-14cm]

          \draw [fill=green!20] (0,0) -- (0,2) -- (1,2) -- (1,4) -- (3,4) -- (3,6) -- (5,6) -- (5,7) -- (7,7) -- (7,0) -- cycle;

          \draw [fill=red!20] (0,0) rectangle +(1,2);

          \draw (0,1) -- (7,1); \draw (1,3) -- (1,0);
          \draw (1,2) -- (7,2); \draw (2,4) -- (2,0);
          \draw (1,3) -- (7,3); \draw (3,5) -- (3,0);
          \draw (3,4) -- (7,4); \draw (4,6) -- (4,0);
          \draw (3,5) -- (7,5); \draw (5,7) -- (5,0);
          \draw (5,6) -- (7,6); \draw (6,7) -- (6,0);

          \draw [very thick] (0,0) rectangle +(2,2);
          \draw [very thick] (1,1) rectangle +(3,3);
          \draw [very thick] (3,3) rectangle +(3,3);
          \draw [very thick] (5,5) rectangle +(2,2);
        \end{scope}
      \end{tikzpicture}
    } 
    \subfloat[1st reduction.]
    {
      \begin{tikzpicture}
        [scale=0.4,
        y=-1cm]

        \begin{scope}
          [xshift=-2cm,yshift=-14cm]

          % The old (1, 1) triangular block
          % \draw [fill=green!20] (0,0) -- (1,1) -- (1,4) -- (3,4) -- (3,6) -- (5,6) -- (5,7) -- (7,7) -- (7,0) -- cycle;
          % The new (1, 1) rectangular block
          \draw [fill=green!20] (0,0) -- (0, 1) -- (1,1) -- (1,4) -- (3,4) -- (3,6) -- (5,6) -- (5,7) -- (7,7) -- (7,0) -- cycle;

          \draw [fill=red!20] (1,1) rectangle +(2,3);

          \draw (1,1) -- (7,1); \draw (1,3) -- (1,0);
          \draw (1,2) -- (7,2); \draw (2,4) -- (2,0);
          \draw (1,3) -- (7,3); \draw (3,5) -- (3,0);
          \draw (3,4) -- (7,4); \draw (4,6) -- (4,0);
          \draw (3,5) -- (7,5); \draw (5,7) -- (5,0);
          \draw (5,6) -- (7,6); \draw (6,7) -- (6,0);

          \fill [pattern=north east lines, pattern color=black!50] (0,0) rectangle +(7,2);

          % The new (1, 1) rectangular block, thick border.  
          \draw [very thick] (0,0) rectangle +(1,1);

          \draw [very thick] (1,1) rectangle +(3,3);
          \draw [very thick] (3,3) rectangle +(3,3);
          \draw [very thick] (5,5) rectangle +(2,2);
        \end{scope}
      \end{tikzpicture}
    }
    \subfloat[2nd reduction.]
    {
      \begin{tikzpicture}
        [scale=0.4,
        y=-1cm]

        \begin{scope}
          [xshift=-2cm,yshift=-14cm]

          % The old (1, 1) triangular block
          % \draw [fill=green!20] (0,0) -- (1,1) -- (1,3) -- (3,3) -- (3,6) -- (5,6) -- (5,7) -- (7,7) -- (7,0) -- cycle;
          % The new (1, 1) rectangular block
          \draw [fill=green!20] (0,0) -- (0,1) -- (1,1) -- (1,3) -- (3,3) -- (3,6) -- (5,6) -- (5,7) -- (7,7) -- (7,0) -- cycle;

          \draw [fill=red!20] (3,3) rectangle +(2,3);

          \draw (1,1) -- (7,1); \draw (1,3) -- (1,0);
          \draw (1,2) -- (7,2); \draw (2,3) -- (2,0);
          \draw (1,3) -- (7,3); \draw (3,5) -- (3,0);
          \draw (3,4) -- (7,4); \draw (4,6) -- (4,0);
          \draw (3,5) -- (7,5); \draw (5,7) -- (5,0);
          \draw (5,6) -- (7,6); \draw (6,7) -- (6,0);

          \fill [pattern=north east lines, pattern color=black!50] (1,1) rectangle +(6,3);

          % The new (1, 1) rectangular block, thick border.  
          \draw [very thick] (0,0) rectangle +(1,1);

          \draw [very thick] (1,1) rectangle +(2,2);
          \draw [very thick] (3,3) rectangle +(3,3);
          \draw [very thick] (5,5) rectangle +(2,2);
        \end{scope}
      \end{tikzpicture}
    }
    \subfloat[3rd reduction.]
    {
      \begin{tikzpicture}
        [scale=0.4,
        y=-1cm]

        \begin{scope}
          [xshift=-2cm,yshift=-14cm]

          % The old (1, 1) triangular block
          % \draw [fill=green!20] (0,0) -- (1,1) -- (1,3) -- (3,3) -- (3,5) -- (5,5) -- (5,7) -- (7,7) -- (7,0) -- cycle;
          % The new (1, 1) rectangular block
          \draw [fill=green!20] (0,0) -- (0,1) -- (1,1) -- (1,3) -- (3,3) -- (3,5) -- (5,5) -- (5,7) -- (7,7) -- (7,0) -- cycle;

          \draw (1,1) -- (7,1); \draw (1,3) -- (1,0);
          \draw (1,2) -- (7,2); \draw (2,3) -- (2,0);
          \draw (1,3) -- (7,3); \draw (3,5) -- (3,0);
          \draw (3,4) -- (7,4); \draw (4,5) -- (4,0);
          \draw (3,5) -- (7,5); \draw (5,7) -- (5,0);
          \draw (5,6) -- (7,6); \draw (6,7) -- (6,0);

          \fill [pattern=north east lines, pattern color=black!50] (3,3) rectangle +(4,3);

          % The new (1, 1) rectangular block, thick border.  
          \draw [very thick] (0,0) rectangle +(1,1);

          \draw [very thick] (1,1) rectangle +(2,2);
          \draw [very thick] (3,3) rectangle +(2,2);
          \draw [very thick] (5,5) rectangle +(2,2);
        \end{scope}
      \end{tikzpicture}
    }

    \caption{Successive reduction of $B$ to block triangular form.
      The diagonal patterns show what has been modified from the previous configuration.
      The thick lines aim to clarify the block structure.
      The red regions identify the sub-matrix of $B$ that will be reduced in the next step.} \label{fig:accelqr2}
  \end{figure}

  In Figure~\ref{fig:accelrq1}~(a) we assumed that the initial shape of $B$ is upper triangular.
  This will be the case only for the first absorption.
  In all subsequent absorptions, the initial shape of $B$ will be as illustrated in Figure~\ref{fig:accelqr2}~(d): when $\ell = 3$, the top-left block may have dimension $p \times p$ with $0 < p \leq 2\nb$, while all the remaining diagonal blocks will be $2\nb \times 2\nb$.
  The first step in the reduction of $V_{2}$ will therefore have to be aligned to respect the block structure of $B$, just as it was the case with the first step of the reduction of $U_2$.

\end{paragraph}

\begin{paragraph}{Handling of iterative refinement failures}
  
  Ideally, reflectors are absorbed only after $k = \nb$ reflectors have been accumulated, i.e., never earlier due to iterative refinement failures.
  In practice, however, failures will occur and as a consequence the details of the procedure described above will need to be adjusted slightly.
  Suppose that iterative refinement fails after accumulating $k < \nb$ reflectors.
  The input matrix $B$ will be (either triangular or) block triangular with diagonal blocks of size $2 \nb \times 2 \nb$ (again, we discuss only the case $\ell = 3$).
  The matrix $V_{2}$ (which has $k$ columns) is reduced using sub-matrices (normally) consisting of $2 \nb + k$ rows.
  The effect on $B$ (\change{cf.~}Figure~\ref{fig:accelrq1}) will be to grow the diagonal blocks from $2 \nb$ to $2 \nb + k$.
  The first $k$ columns of these diagonal blocks are then reduced just as before (\change{cf.~}Figure~\ref{fig:accelrq2}) but this time the RQ decompositions will be computed from sub-matrices of size $2 \nb \times (2 \nb + k)$, i.e., from sub-matrices with $\nb - k$ fewer columns than before.
  Note that the final WY transformations will involve only $k$ reflectors (instead of $\nb$), which is important for the sake of efficiency.
  Similarly, when reducing $U_{2}$ the sub-matrices normally consist of $2 \nb + k$ rows and the diagonal blocks of $B$ will grow by $k$ once more (\change{cf.~}Figure~\ref{fig:accelqr1}).
  The block triangular structure of $B$ is finally restored by transformations consisting of $k$ reflectors (\change{cf.~}Figure~\ref{fig:accelqr2}).

\end{paragraph}

\begin{paragraph}{Impact on Algorithm~\ref{alg:main}}
  
  The impact of the block triangular form in Figure~\ref{fig:accelqr2}~(d) on Algorithm~\ref{alg:main} is minor.
  Aside from modifying the way in which reflectors are absorbed (as described above), the only other necessary change is to modify the implicit reduction of column $j + 1$ of $B$ to accommodate a \emph{block} triangular matrix.
  In particular, the residual computation will involve multiplication with a \emph{block} triangular matrix instead of a triangular matrix and the solve will require \emph{block} backwards substitution instead of regular backwards substitution.
  The block backwards substitution is carried out by computing an LU decomposition (with partial pivoting) once for each diagonal block and then reusing the decompositions for each of the (up to) $k$ solves leading up to the next wave of absorption.

  \change{The optimal value of $\ell$ is a trade-off between efficiency in the reduction of $V_2$ and $U_2$, and efficiency in solving linear systems with $B$ having large diagonal blocks. A larger value of $\ell$ helps speed up the former, but results in increased fill-in in $B$, as well as in increased time needed for computing LU factorizations of its large diagonal blocks.}

\end{paragraph}

%!TEX root = 0__main.tex

\section{Numerical Experiments} \label{sec:experiments}

To test the performance of our newly proposed HouseHT algorithm, we implemented it in C++ and executed it on two different machines using different BLAS implementations.
We compare with the LAPACK routine DGGHD3, which implements the block-oriented Givens-based algorithm from~\cite{Kagstrom2008} and can be considered state of the art, as well as the predecessor LAPACK routine DGGHRD, which implements the original Givens-based algorithm from~\cite{Moler1973}.
We created four test suites in order to explore the behavior of the new algorithm on a wide range of matrix pencils.
For each test pair, the correctness of the output was verified by checking the resulting matrix structure and by computing $\|{H} - {Q}^T A {Z}\|_F$ and $\|{T} - {Q}^T B {Z}\|_F$. %, where $A$ and $B$ are the input matrices, and $\tilde{Q}, \tilde{Z}, \tilde{H}, \tilde{T}$ the algorithm output.

\change{Table \ref{tab:computingenvironments}} describes the computing environments used in our tests. The last row illustrates the relative performance of the machine/BLAS combinations,
measuring the timing of the DGGHD3 routine for a random pair of dimension $4000$, and rescaling so that the time for pascal with MKL is normalized to $1.00$.

\begin{table}
    \caption{\label{tab:computingenvironments} Computing environments used for testing.}
    \begin{center} \scriptsize 
	\begin{tabular}{|c|c|c|c|c|}\hline
		machine name     & \multicolumn{2}{c|}{pascal}                                            & \multicolumn{2}{c|}{kebnekaise} \\\hline
		processor        & \multicolumn{2}{c|}{2x Intel Xeon E5-2690v3                          } & \multicolumn{2}{c|}{2x Intel Xeon E5-2690v4} \\
		                 & \multicolumn{2}{c|}{                        (12 cores each, $2.6$GHz)} & \multicolumn{2}{c|}{(14 cores each, $2.6$GHz)} \\\hline
		RAM              & \multicolumn{2}{c|}{$256$GB}                                           & \multicolumn{2}{c|}{$128$GB} \\\hline
		operating system & \multicolumn{2}{c|}{Centos $7.3$}                                      & \multicolumn{2}{c|}{Ubuntu $16.04$} \\\hline
		BLAS library     & MKL $11.3.3$                                                           & OpenBLAS $0.2.19$                        & MKL $2017.3.196$ & OpenBLAS 0.2.20 \\\hline
		compiler         & icpc $16.0.3$                                                          & g++ $4.8.5$                              & g++ $6.4.0$  & g++ $6.4.0$ \\\hline
        \change{\newchange{DGGHD3} parameters} & \change{$nb=144$} & \change{$nb=192$} & \change{$nb=128$} & \change{$nb=176$} \\\hline
        \change{\newchange{HouseHT} parameters} & \change{$nb=128$, $\ell=4$} & \change{$nb=128$, $\ell=4$} & \change{$nb=64$, $\ell=5$} & \change{$nb=80$, $\ell=4$} \\\hline
        relative timing  & $1.00$                                                                 & 1.38                                     & 0.77 & 0.88 \\ \hline
	\end{tabular}
    \end{center}
\end{table}

For each computing environment, the optimal block sizes for HouseHT and DGGHD3 were first estimated empirically and then used in all four test suites; \change{Table~\ref{tab:computingenvironments} lists those that turned out to be optimal for larger matrices ($n \geq 3500$).}
Unless otherwise stated, we use only a single core and link to single-threaded BLAS.
All timings include the accumulation of orthogonal transformations into $Q$ and $Z$.
\\

\begin{paragraph}{Test Suite 1: Random matrix pencils}
 The first test suite consists of random matrix pencils. More specifically, the matrix $A$ has normally distributed entries while the matrix $B$ is chosen as the triangular factor of the QR decomposition of a matrix with normally distributed entries.
This test suite is designed to illustrate the behavior of the algorithm for a ``non-problematic'' input with no infinite eigenvalues and a fairly well-conditioned matrix $B$.
For such inputs, the HouseHT algorithm typically needs no iterative refinement steps when solving linear systems.

Figure \ref{fig:suite1-fig1} displays the execution time of HouseHT divided by the execution  time of DGGHD3 for the different computing environments.
The new algorithm has roughly the same performance as DGGHD3, being from about $20\%$ faster to about $35\%$ slower than DGGHD3, depending on the machine/BLAS combination. 
Both algorithms exhibit far better performance than the LAPACK routine DGGHRD, which makes little use of \change{level-3 BLAS} due to its non-blocked nature. 

Figure \ref{fig:suite1-fig2} shows the flop-rates of HouseHT and DGGHD3 for the pascal machine with MKL BLAS.
Although the running times are about the same, the new algorithm computes about twice as many floating point operations, so the resulting
flop-rate is about two times higher than DGGHD3. The flop-counts were obtained during the execution of the algorithm
by interposing calls to the LAPACK and BLAS routines and instrumenting the code.

\begin{figure}[H]
    \begin{minipage}{1.0\textwidth}
        \centering
        \subfloat[Execution time of HouseHT and DGGHRD relative to execution  time of DGGHD3.]
            {\includegraphics[width=.47\textwidth]{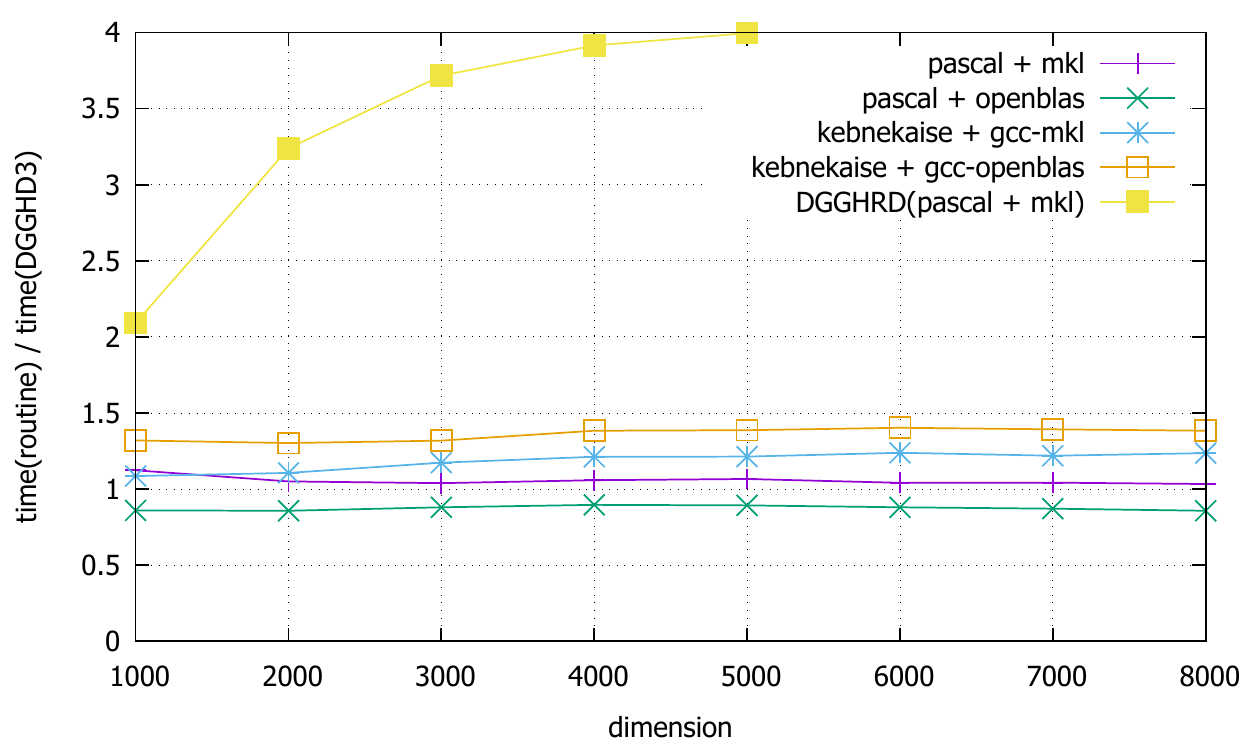}
            \label{fig:suite1-fig1}}
        \hfill
        \subfloat[Flop-rate of HouseHT and DGGHD3 on the pascal machine with MKL BLAS.]
            {\includegraphics[width=.47\textwidth]{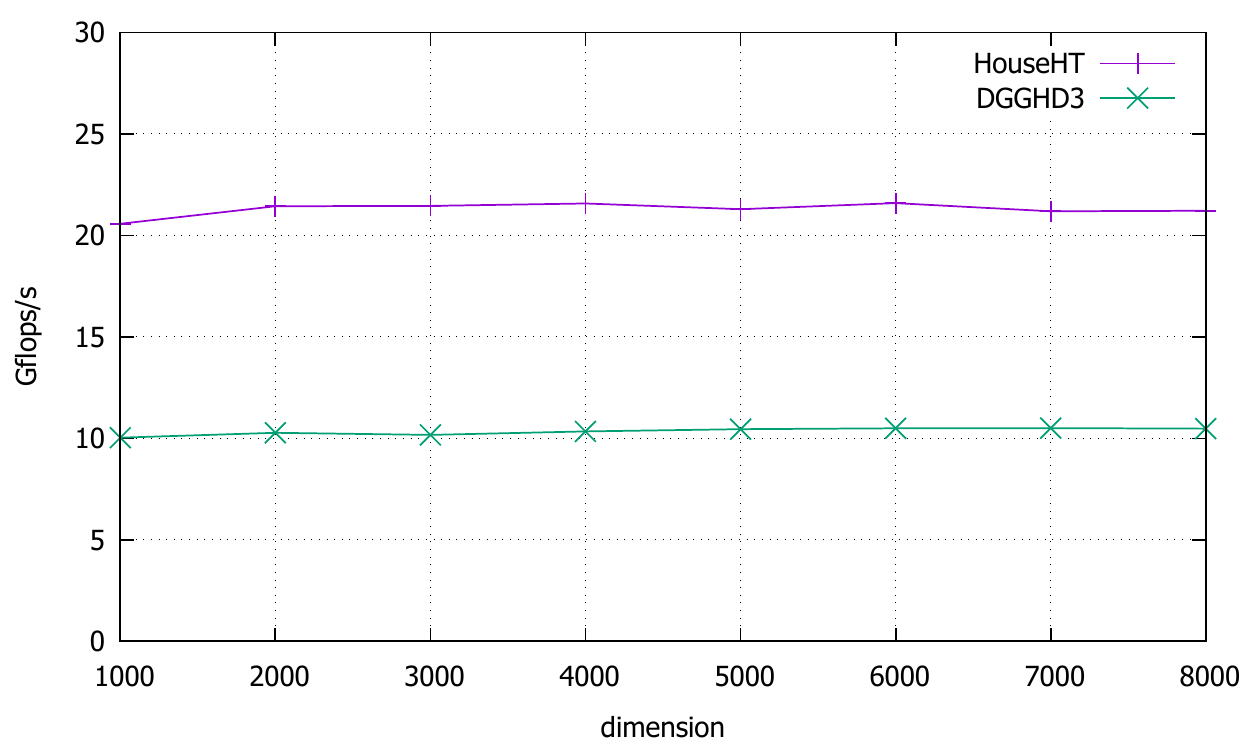}
            \label{fig:suite1-fig2}}

        \caption{Single-core performance of HouseHT for randomly generated matrix pencils (Test Suite~1).}
    \end{minipage}
\end{figure}

\change{Table~\ref{tab:subroutines}} shows the fraction of the time that HouseHT spends in the three most computationally expensive parts of the algorithm. 
The results are from the pascal machine with MKL BLAS and $n=8000$. 

% \noindent
\begin{table}[H]
    \caption{\label{tab:subroutines} Relative timings of the HouseHT subroutines.}
    \begin{minipage}{1.0\textwidth}
        \centering
        % machine = pascal, blas = mkl
% N=8000
\begin{tabular}{|c|c|}\hline
    part of HouseHT                                           & \% of total time \\ \hline\hline
    solving systems with $B$, computing residuals             & 22.82\% \\ \hline
    absorption of reflectors                                  & 57.40\% \\ \hline
    assembling $Y=AVT$                                        & 19.61\% \\ \hline
\end{tabular}
    
    \end{minipage}
\end{table}

% \ \\

\noindent
\begin{minipage}{1.0\textwidth}
    % machine = pascal, blas = mkl
% N=8000
HouseHT spends as much as 92.60\% of its flops
(and 52.77\% of its time)
performing level 3 BLAS operations,
compared to DGGHD3 which spends only 65.35\% of its flops
(and 18.33\% of its time) in level 3 BLAS operations.
    
\end{minipage}

\end{paragraph}

\begin{paragraph}{Test Suite 2: Matrix pencils from benchmark collections}
The purpose of the second test suite is to demonstrate the performance of HouseHT for matrix pencils originating from a variety of applications. To this end, we applied HouseHT and DGGHD3 to a number of pencils from the 
benchmark collections \cite{MMarket2007, Chahlaoui2002,Korvink2005}.
Table~\ref{tab:benchmarks} displays the obtained results for the pascal machine with MKL BLAS. 
When constructing the Householder reflector for reducing a column of $B$ in HouseHT, the percentage of columns that require iterative refinement varies strongly for the different examples. Typically, at most one or two steps of iterative refinement are necessary to achieve numerical stability. It is important to note that we did not observe a single failure, all linear systems were successfully solved in less than $10$ iterations. %In turn,  and thus block reflectors in all examples were consumed only due to reaching maximum size.

As can be seen from Table~\ref{tab:benchmarks}, HouseHT brings little to no benefit over DGGHD3 on a single core of pascal with MKL. A first indication of the benefits HouseHT may bring for several cores is seen by comparing the third and the fourth columns of the table. By switching to multithreaded BLAS and using eight cores, then for sufficiently large matrices HouseHT becomes significantly faster than DGGHD3.

\change{The percentage} of columns for which an extra IR step is required depends slightly on the machine/BLAS combination due to different block size configurations; typically, it does not differ by much, and difficult examples remain difficult. \change{(The fifth column of Table~\ref{tab:benchmarks} may be used here as the difficulty indicator.)}
    The performance of HouseHT \change{vs.~}DGGDH3 does vary more, as Figure \ref{fig:suite1-fig1} suggests. We briefly summarize the findings of the numerical experiments:
    when the algorithms are run on a single core, the ratios shown in the \change{third} column of \change{Table~\ref{tab:benchmarks}} are, on average, about 20\% smaller for pascal/OpenBLAS, about 5\% larger for kebnekaise/MKL, and about 28\% larger for kebnekaise/OpenBLAS. 
    When the algorithms are run on $8$ cores, the HouseHT algorithm \change{increasingly outperforms} DGGHD3 with the increasing matrix size, regardless of the machine/BLAS combination. On average, the ratios shown in the \change{fourth} column are about 38\% smaller for pascal/OpenBLAS, about 14\% larger for kebnekaise/OpenBLAS, and about 50\% larger for kebnekaise/MKL.
    \\

\begin{table}
 \caption{\label{tab:benchmarks} Execution time of HouseHT relative to DGGHRD for various benchmark examples (Test Suite~2), on a single core and on eight cores (\change{pascal/MKL}). }
\centering
    % machine = pascal, blas = mkl
% multicore machine = pascal, blas = mkl
\begin{tabular}{|c|c|c|c|c|c|}\hline name & $n$ & \parbox{2.37cm}{\small time(HouseHT)/ time(DGGHD3) (1 core)} & \parbox{2.37cm}{\small time(HouseHT)/ time(DGGHD3) (8 cores)} & \parbox{1.5cm}{\small \%\! columns with extra IR steps} & \parbox{1.5cm}{\small avg.~\#IR steps per column} \\ \hline
	BCSST20         & $  485$ & $ 1.30$ & $ 1.36$ & $52.58$ & $ 0.52$ \\
	MNA\_1          & $  578$ & $ 1.04$ & $ 1.31$ & $42.39$ & $ 1.02$ \\
	BFW782          & $  782$ & $ 1.18$ & $ 0.90$ & $ 0.00$ & $ 0.00$ \\
	BCSST19         & $  817$ & $ 0.98$ & $ 1.03$ & $55.57$ & $ 0.55$ \\
	MNA\_4          & $  980$ & $ 1.05$ & $ 0.91$ & $34.39$ & $ 0.42$ \\
	BCSST08         & $ 1074$ & $ 1.11$ & $ 0.99$ & $15.08$ & $ 0.15$ \\
	BCSST09         & $ 1083$ & $ 1.13$ & $ 0.93$ & $43.49$ & $ 0.43$ \\
	BCSST10         & $ 1086$ & $ 1.17$ & $ 0.85$ & $16.94$ & $ 0.17$ \\
	BCSST27         & $ 1224$ & $ 1.11$ & $ 0.74$ & $24.43$ & $ 0.24$ \\
	RAIL            & $ 1357$ & $ 1.03$ & $ 0.71$ & $ 0.52$ & $ 0.00$ \\
	SPIRAL          & $ 1434$ & $ 1.04$ & $ 0.68$ & $ 0.00$ & $ 0.00$ \\
	BCSST11         & $ 1473$ & $ 1.05$ & $ 0.67$ & $ 7.81$ & $ 0.08$ \\
	BCSST12         & $ 1473$ & $ 1.03$ & $ 0.67$ & $ 1.29$ & $ 0.01$ \\
	FILTER          & $ 1668$ & $ 1.03$ & $ 0.62$ & $ 0.36$ & $ 0.00$ \\
	BCSST26         & $ 1922$ & $ 1.05$ & $ 0.58$ & $20.29$ & $ 0.20$ \\
	BCSST13         & $ 2003$ & $ 1.05$ & $ 0.59$ & $26.21$ & $ 0.28$ \\
	PISTON          & $ 2025$ & $ 1.06$ & $ 0.57$ & $20.79$ & $ 0.27$ \\
	BCSST23         & $ 3134$ & $ 1.19$ & $ 0.56$ & $72.59$ & $ 0.73$ \\
	MHD3200         & $ 3200$ & $ 1.16$ & $ 0.54$ & $26.97$ & $ 0.27$ \\
	BCSST24         & $ 3562$ & $ 1.19$ & $ 0.54$ & $46.97$ & $ 0.47$ \\
	BCSST21         & $ 3600$ & $ 1.11$ & $ 0.48$ & $11.53$ & $ 0.11$ \\
	\hline
\end{tabular}
\end{table}
\end{paragraph}

\begin{paragraph}{Test Suite 3: Potential for parallelization}
The purpose of the third test is a more detailed exploration of the potential benefits the new algorithm may achieve in a parallel environment. For this purpose, we link HouseHT with a multithreaded BLAS library. Let us emphasize that this is purely \change{exploratory}.
\change{There are possibilities for parallelism outside of BLAS invocations; this is subject to future work.}
Figure \ref{fig:suite2-fig1} shows the speedup of the HouseHT algorithm achieved relative to DGGHD3 for an increasing number of cores.
We have used \change{$8000\times 8000$} matrix pencils, generated as in Test Suite~1.
As shown in Figure \ref{fig:suite2-fig2}, the
performance of DGGHD3, unlike the new algorithm, barely benefits from switching to multithreaded BLAS.
\end{paragraph}

\begin{figure}[H]
    \begin{minipage}{1.0\textwidth}
        \centering
        \subfloat[Execution time of DGGHD3 relative to HouseHT.]
            {\includegraphics[width=.47\textwidth]{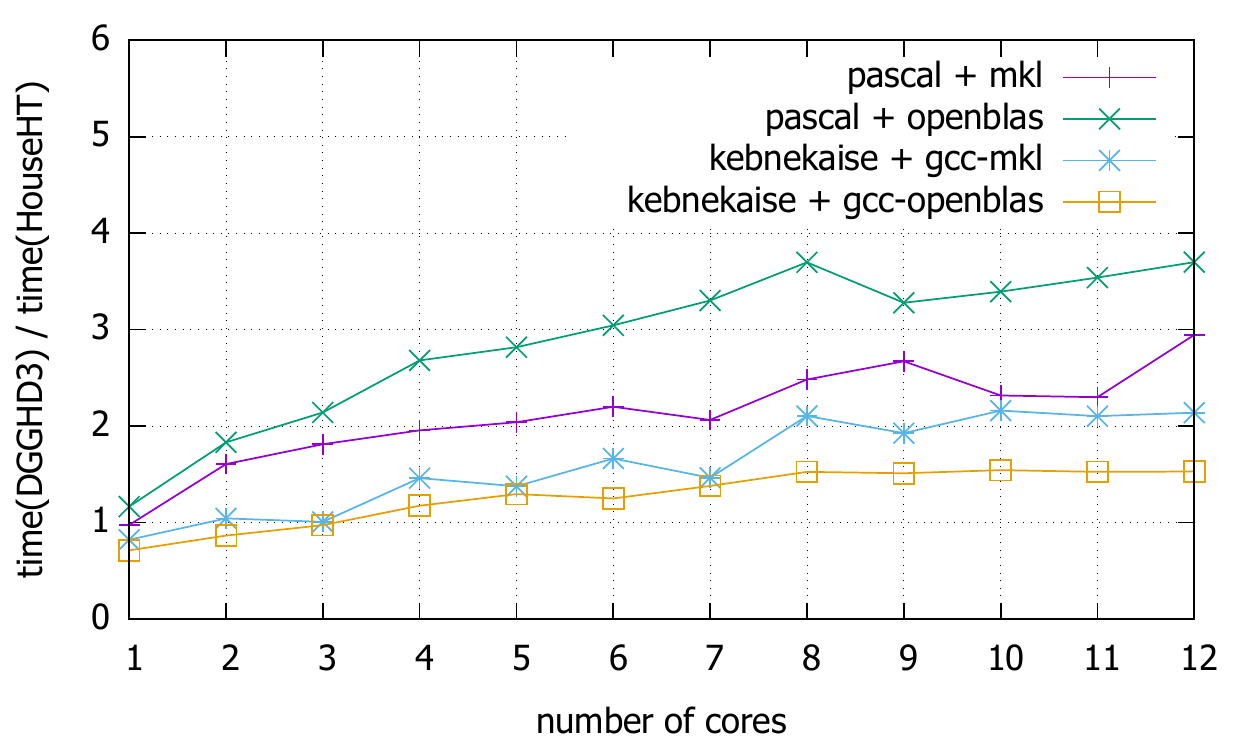}
            \label{fig:suite2-fig1}}
        \hfill
        \subfloat[Execution times of HouseHT and DGGHD3 on several cores relative to the single-core execution time (pascal+MKL).]
            {\includegraphics[width=.47\textwidth]{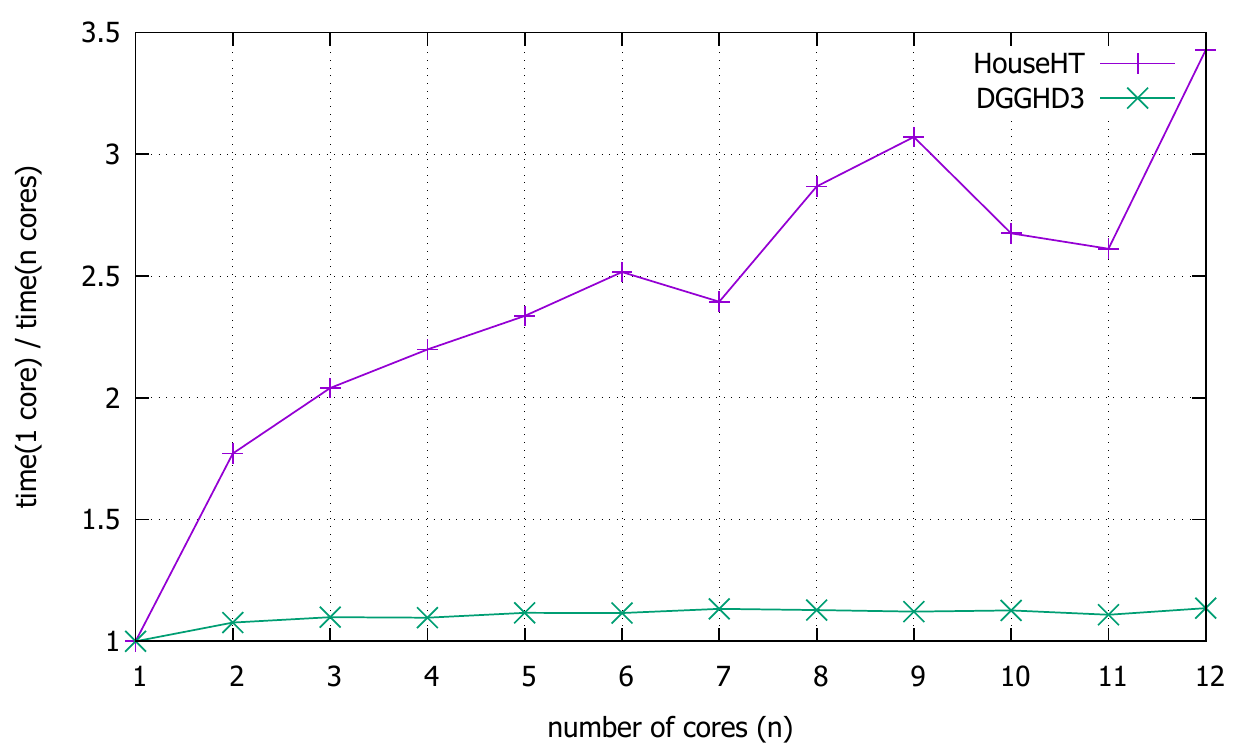}
            \label{fig:suite2-fig2}}

        \caption{Performance of HouseHT for randomly generated matrix pencils of dimension $8000$ (Test Suite~3) when using multithreaded BLAS.}
    \end{minipage}
\end{figure}

\begin{paragraph}{Test Suite 4: Saddlepoint matrix pencils}
 The final test suite consists of matrix pencils designed to be particularly unfavorable for HouseHT.
Let
$$
  A-\lambda B = \left[ \begin{array}{cc} X & Y \\ Y^T & 0 \end{array}\right]- \lambda \left[ \begin{array}{cc} I & 0 \\ 0 & 0 \end{array}\right],
$$
with a random \change{symmetric} positive definite matrix $X$ and a random (full-rank) matrix $Y$ with sizes chosen such that $X$ is $3/4$th the \change{dimension} of $A$.
The matrix $B$ is split accordingly.
For such matrix pencils, with many infinite eigenvalues, we expect that HouseHT will struggle with solving linear systems, requiring many steps of iterative refinement
and being forced to prematurely absorb reflectors.
This is, up to a point, what happens when we run the test suite. In Table~\ref{tab:infinite}, we see that HouseHT may be up to $4$ times slower than DGGHD3 (on pascal/MKL) for smaller-sized matrix pencils. For about $5\%$ of the columns the linear systems cannot be solved in a stable manner, even with the
help of iterative refinement. In turn, the reflectors have to be repeatedly absorbed prematurely. 
However, in all of these cases, HouseHT still manages to successfully produce the Hessenberg-triangular form to full precision.

For example, for $n=4000$, there are $67$ columns for which the linear system cannot be solved with $10$ steps of iterative
refinement. The failure happens more frequently in the beginning of the algorithm: 
it occurs $14$ times within the first $100$ columns, only $6$ times after the $700$th column,
and the last occurrence is at the $1082$nd column. 
The same observation can be made for columns requiring extra (but fewer than 10) steps of IR; the last such column is the $2105$th column.

For this, and many similar test cases, using the preprocessing of the zero columns as described in Section~\ref{sec:preprocessing} 
may convert a difficult test case to a very easy one. The numbers in parentheses in Table~\ref{tab:infinite} show the effect of preprocessing for
the saddlepoint pencils. Note that we do not preprocess the input for DGGHD3 (which would benefit from it as well).
With preprocessing on, there is barely any need for iterative refinement despite the fact that it does not remove all of the infinite eigenvalues.

\begin{table}
\caption{\label{tab:infinite} Execution time of HouseHT relative to DGGHD3 for matrix pencils with saddlepoint structure (Test Suite~4), \change{run on pascal/MKL}. The numbers in parentheses correspond to results obtained with preprocessing.}
	\centering
    % machine = pascal, blas = mkl
% Saddlepoint: A = [X Y; Y' 0], B = [I 0; 0 0] with pos.def. X, full-rank Y, len(X)/len(A)=0.75
% Number before () shows the result without preprocessing the zero columns.
% Number inside () shows the result with preprocessing the zero columns.
\begin{tabular}{|c|c|c|c|c|}\hline
    $n$ & \parbox{2.37cm}{\small time(HouseHT)/ time(DGGHD3)} & \parbox{2.37cm}{\small \% columns with failed IR} & \parbox{2.37cm}{\small \% columns with extra IR steps} & \parbox{2.4cm}{\small average \#IR steps per column} \\ \hline
	$ 1000$ & $ 2.52$ ($ 0.91$) & $ 4.90$ ($ 0.10$) & $27.40$ ($ 0.40$) & $ 2.40$ ($ 0.02$) \\
	$ 2000$ & $ 1.72$ ($ 0.79$) & $ 1.60$ ($ 0.00$) & $32.55$ ($ 0.80$) & $ 1.95$ ($ 0.06$) \\
	$ 3000$ & $ 2.01$ ($ 0.77$) & $ 1.73$ ($ 0.00$) & $35.20$ ($ 0.77$) & $ 2.06$ ($ 0.07$) \\
	$ 4000$ & $ 2.14$ ($ 0.78$) & $ 1.55$ ($ 0.00$) & $37.10$ ($ 0.72$) & $ 2.06$ ($ 0.06$) \\
	$ 5000$ & $ 2.03$ ($ 0.78$) & $ 1.24$ ($ 0.00$) & $35.70$ ($ 0.72$) & $ 1.86$ ($ 0.06$) \\
	$ 6000$ & $ 1.91$ ($ 0.77$) & $ 0.77$ ($ 0.00$) & $38.00$ ($ 0.62$) & $ 1.83$ ($ 0.01$) \\
	$ 7000$ & $ 1.89$ ($ 0.76$) & $ 0.77$ ($ 0.00$) & $35.76$ ($ 0.67$) & $ 1.72$ ($ 0.05$) \\
	$ 8000$ & $ 1.78$ ($ 0.77$) & $ 0.54$ ($ 0.00$) & $36.30$ ($ 0.66$) & $ 1.65$ ($ 0.04$) \\
	\hline
\end{tabular}
\end{table}
\end{paragraph}

%!TEX root = 0__main.tex

\section{Conclusions and future work}
\label{sec:conclusions}

We described a new algorithm for Hessen\-berg-triangular reduction.
The algorithm relies on the unconventional and little-known possibility to use a Householder reflector applied from the \emph{right} to reduce a matrix \emph{column}~\cite{Watkins2000}.
In contrast, the current state of the art is entirely based on Givens rotations~\cite{Kagstrom2008}.

We \change{observed the algorithm to be numerically backward stable} but its performance may degrade when presented with a difficult problem.
Extensive experiments on synthetic as well as real examples suggest that performance degradation is not a significant concern in practice and that simple preprocessing measures can be applied to greatly reduce the negative effects.

Compared with the state of the art~\cite{Kagstrom2008}, the new algorithm requires a small constant factor more floating point arithmetic operations but on the other hand these operations occur in computational patterns that allow for faster flop rates \change{(i.e., higher arithmetic intensity)}.
In other words, the negative impact of the additional flops is at least partially counteracted by the increased speed by which these flops can be performed.
Experiments suggest that the sequential performance of the new algorithm is comparable to the state of the art.

The primary motivation for developing the new algorithm was its potential for greater parallel scalability than the state-of-the-art parallel algorithm~\cite{Bjoern16}.
Early experiments using \change{multithreaded} BLAS support this idea.
Therefore, the design and evaluation of a task-based parallel algorithm is our next step.

\section*{Acknowledgments}

This research was conducted using the resources of High Performance Computing Center North (HPC2N).
Part of this work was done while the first author was a postdoctoral researcher at \'{E}cole polytechnique f\'{e}d\'{e}rale de Lausanne, Switzerland.

\bibliographystyle{plain}
\bibliography{anchp}

\appendix 

%!TEX root = 0__main.tex

% \color{red}

\section{Error analysis of basic Householder-based Hessenberg-tri\-an\-gu\-lar reduction} \label{app:error}

In this section, we perform an error analysis of the basic algorithm outlined in the beginning of Section~\ref{sec:overview}. For this purpose, we provide a formal description in Algorithm~\ref{alg:basichtred}. Given a vector $y$ of length $k$ the function {\tt House}$(y)$ used in Algorithm~\ref{alg:basichtred} returns a Householder reflector $F$ such that the last $k-1$ entries of $Fy$ are zero.

\begin{algorithm2e}[htbp]  %\color{red} 
  \caption{$[H, T, Q, Z] = \mathtt{BasicHouseHT}(A, B)$}
  \label{alg:basichtred}
  \tcp{Initialize}
  $Q \gets I_n$, $Z \gets I_n$\;
  \tcp{For each column to reduce in $A$}
  \For{$j = 1:n-2$}
  {
    \tcp{Reduce column $j$ of $A$}
    Construct Householder reflector $F_j = I_j \oplus {\tt House}(A_{j+1:n,j})$\; \label{step:lefthh}
    
    Update $A \gets F_j A $, $B \gets F_j B$, $Q\gets Q F_j$\; \label{step:applylefthh}
    
    Set $A_{j+2:n,j} \gets 0$\; \label{step:setAzero}
    
    \tcp{Reduce column $j+1$ of $B$}
    Solve linear system $B_{j+1:n,j+1:n} x = e_{1}$\; \label{step:solveB}
    
    Construct Householder reflector $G_j = I_j\oplus {\tt House}(x)$\;
    
    Update $A \gets A G_j$, $B \gets B G_j$, $Z \gets Z G_j$\;

    Set $B_{j+2:n,j+1} \gets 0$\; \label{step:setBzero}
  }
  \Return $[A, B, Q, Z]$\;
\end{algorithm2e}

In the following analysis, we will assume that Step~\ref{step:solveB} of Algorithm~\ref{alg:basichtred} is carried out in a backward stable manner. To be more specific, letting $\tilde B^{(j)}$ denote the computed matrix $B$ in loop $j$ of Algorithm~\ref{alg:basichtred} after Step~\ref{step:applylefthh}, we assume that the computed vector $\hat x^{(j)}$ satisfies
\begin{equation} \label{eq:computedx}
 \big( \tilde B^{(j)} + \triangle_S^{(j)} \big) \hat x^{(j)} = e_1,\qquad \big\|\triangle_S^{(j)}\big\|_F = \mathcal O(\mathrm u)\cdot \big\| \tilde B^{(j)}\big\|_F.
\end{equation}
Here and in the following, the constant in $\mathcal O(\mathrm u)$ depends on $n$ only and grows mildly with $n$.
 
\begin{theorem} \label{thm:htredanalysis}
Assume that~\eqref{eq:computedx} holds and let $\hat H, \hat T$ denote the matrices returned by Algorithm~\ref{alg:basichtred} carried out in floating point arithmetic according to the standard model~\cite[Eq.~(2.4)]{Higham2002}. Then $\hat H$ is upper Hessenberg, $\hat T$ is upper triangular, and there are orthogonal matrices $Q,Z$ such that
\begin{equation} \label{eq:perturbationshtred}
 \hat H = Q^T (A+\triangle_A) Z, \quad \hat T = Q^T (B+\triangle_B) Z,
\end{equation}
where $\triangle_A, \triangle_B$ satisfy
\[
 \|\triangle A\|_F = \mathcal O(\mathrm u)\cdot \| A \|_F, \quad \|\triangle B\|_F = \mathcal O(\mathrm u)\cdot \| B \|_F.
\]
\end{theorem}
\begin{proof}
 Let $\hat A^{(j)}$, $\hat B^{(j)}$ denote the computed matrices $A,B$ after $j$ loops of Algorithm~\ref{alg:basichtred} have been completed. For every $j = 0,1,\ldots,n-2$ we claim that there exist orthogonal matrices
 $Q^{(j)}, Z^{(j)}$ and perturbations $\triangle_A^{(j)}$, $\triangle_B^{(j)}$ satisfying
 $
 \|\triangle A^{(j)}\|_F = \mathcal O(\mathrm u)\cdot \| A \|_F$, $\|\triangle B^{(j)}\|_F = \mathcal O(\mathrm u)\cdot \| B \|_F$
 such that
 \begin{equation} \label{eq:inductionhypo}
 \hat A^{(j)} = \big( Q^{(j)} )^T \big(A+\triangle_A^{(j)}\big) Z^{(j)}, \quad \hat B^{(j)} = \big( Q^{(j)} )^T \big( B+\triangle_B^{(j)}\big) Z^{(j)}.
 \end{equation}
Because of $\hat H = \hat A^{(n-2)}$, $\hat T = \hat B^{(n-2)}$, this claim implies~\eqref{eq:perturbationshtred}.

We prove~\eqref{eq:inductionhypo} by induction. For $j = 0$, this relation is trivially satisfied. Suppose now that it holds for $j-1$. Considering the $j$th loop of Algorithm~\ref{alg:basichtred}, we first treat Step~\ref{step:lefthh} and let $F_j$ denote the exact Householder reflector constructed from the (perturbed) matrix $\hat A^{(j-1)}$. Letting $\tilde A^{(j)}$, $\tilde B^{(j)}$ denote the computed matrices $A,B$ after Step~\ref{step:applylefthh}, Lemma 19.2 from \cite{Higham2002} implies that there exist
$\triangle_{A,1}$, $\triangle_{B,1}$ with \begin{eqnarray*}
 &&\|\triangle_{A,1}\|_F = \mathcal O(\mathrm u)\cdot \|\hat A^{(j-1)} \|_F = \mathcal O(\mathrm u)\cdot \|A\|_F, \\
 &&\|\triangle_{B,1}\|_F = \mathcal O(\mathrm u)\cdot \|\hat B^{(j-1)} \|_F = \mathcal O(\mathrm u)\cdot \|B\|_F,
\end{eqnarray*}
such that
\begin{equation} \label{eq:pertafterleft}
 \tilde A^{(j)} = F_j \big(\hat A^{(j-1)} + \triangle_{A,1} \big),\quad 
  \tilde B^{(j)} = F_j \big(\hat B^{(j-1)} + \triangle_{B,1} \big)
\end{equation}
This statement remains true after Step~\ref{step:setAzero}; see the proof of Theorem 19.4 in~\cite{Higham2002}.
Analogously, we have after Step~\ref{step:setBzero} that
\begin{equation} \label{eq:pertafterright}
  \hat A^{(j)} = \big(\tilde A^{(j)} + \triangle_{A,2} \big) G_j ,\quad 
   \hat B^{(j)} = \big(\tilde B^{(j)} + \triangle_{B,2} \big) G_j,
\end{equation}
with 
$\|\triangle_{A,2}\|_F = \mathcal O(\mathrm u)\cdot \|A\|_F$, 
$\|\triangle_{B,2}\|_F = \mathcal O(\mathrm u)\cdot \|B\|_F$. While the existence of $\triangle_{A,2}$ follows again from \cite[Lemma 19.2]{Higham2002}, the existence of $\triangle_{B,2}$ follows from Lemma~\ref{lemma:analysisopposite}.

Combining~\eqref{eq:pertafterleft} and~\eqref{eq:pertafterright} gives
\[
 \hat A^{(j)} = F_j\big( \hat A^{(j-1)} + \triangle_{A,1} +F_j\triangle_{A,2} \big) G_j, \quad
 \hat B^{(j)} = F_j\big( \hat B^{(j-1)} + \triangle_{B,1} +F_j\triangle_{B,2} \big) G_j.
\]
Therefore~\eqref{eq:inductionhypo} holds for $j$ by setting
$Q^{(j)} := Q^{(j-1)}F_j$, 
$Z^{(j)} := Z^{(j-1)}G_j$, and
\begin{eqnarray*}
\triangle_A^{(j)} &:=& \triangle_A^{(j-1)} + Q^{(j-1)}\big( \hat A^{(j-1)} + \triangle_{A,1} +F_j\triangle_{A,2} \big) \big( Z^{(j-1)} \big)^T \\
\triangle_B^{(j)} &:=& \triangle_B^{(j-1)} + Q^{(j-1)}\big( \hat B^{(j-1)} + \triangle_{B,1} +F_j\triangle_{B,2} \big) \big( Z^{(j-1)} \big)^T.
\end{eqnarray*}

It remains to show that $\hat H$, $\hat T$ are upper Hessenberg/triangular, but this follows directly from Steps~\ref{step:setAzero} and~\ref{step:setBzero}, combined with the observations that subsequent operations do not modify entries that have been set to zero.
\end{proof}

Along the line of arguments given in~\cite[Pg.~360]{Higham2002}, the computed transformation matrices $\hat Q, \hat Z$ of Algorithm~\ref{alg:basichtred} satisfy $\|Q-\hat Q\|_F = \mathcal O(\mathrm u)$, $\|Z-\hat Z\|_F = \mathcal O(\mathrm u)$, with the matrices $Q,Z$ from Theorem~\ref{thm:htredanalysis}. On the one hand, this implies that $\hat Q, \hat Z$ are very close to being orthogonal. On the other hand, this implies small residuals:
\begin{eqnarray*}
  \| A -  \hat Q \hat H \hat Z^T \|_F &=&  \| A -  Q \hat H Z^T + (Q-\hat Q) \hat H \hat Z^T + Q\hat H (Z-\hat Z)^T  \|_F \\
  &\le & \|\triangle_A\|_F + \mathcal O(\mathrm u) \cdot \|\hat H\|_F = \mathcal O(\mathrm u) \cdot \|A\|_F,
\end{eqnarray*}
and analogously $\| B -  \hat Q \hat T \hat Z^T \|_F = \mathcal O(\mathrm u) \cdot \|B\|_F$.

\end{document}